\documentclass{amsproc}

\usepackage{amssymb}

\usepackage{graphicx}

\usepackage{pstricks,textcomp}

\newtheorem{theorem}{Theorem}[section]
\newtheorem{lemma}[theorem]{Lemma}
\newtheorem{proposition}[theorem]{Proposition} % ajout

\theoremstyle{definition}
\newtheorem{definition}[theorem]{Definition}

\theoremstyle{remark}
\newtheorem{remark}[theorem]{Remark}

\numberwithin{equation}{section}

\def\hfl#1{\smash{\mathop{\hbox to 10mm{\rightarrowfill}}\limits^{\textstyle
#1}}}
\newcommand{\KK}{\mathbb{K}} 
\newcommand{\CA}{{\mathcal A}} 
\newcommand{\TCM}{\tilde{C}_2(M)}
\newcommand{\TCMD}{\widetilde{M^2}} 
\newcommand{\CE}{{\mathcal E}} 
\newcommand{\CF}{{\mathcal F}} 
\newcommand{\CM}{{\mathcal M}}
\newcommand{\CV}{{\mathcal V}} 
\newcommand{\CQ}{{\mathcal{Q}}}
\newcommand{\CS}{{\mathcal S}}
\newcommand{\fvar}{x}
\newcommand{\svar}{y}
\newcommand{\tvar}{z}
\newcommand{\fvarM}{X}
\newcommand{\svarM}{Y}
\newcommand{\tvarM}{Z}
\newcommand{\qvarM}{W}
\newcommand{\cvarM}{V}
\newcommand{\funcM}{f_M}
\newcommand{\dfuncM}{g_M}
\newcommand{\CfuncM}{C_M}
\newcommand{\projM}{p_M}
\newcommand{\proj}{p}
\newcommand{\projq}{q_S}
\newcommand{\Bor}{B}
\newcommand{\ID}{I_{\Delta}}
\newcommand{\ZZ}{\mathbb{Z}} 
\newcommand{\RR}{\mathbb{R}} 
\newcommand{\QQ}{\mathbb{Q}} 
\newcommand{\CC}{\mathbb{C}} 
\newcommand{\NN}{\mathbb{N}}

\begin{document}

\title[Invariants derived from the equivariant linking pairing]{Invariants of knots and $3$--manifolds \\derived from the equivariant linking pairing}


%    author one information
\author{Christine Lescop}
\address{Institut Fourier, BP 74, 38402 Saint-Martin d'H\`eres cedex, France}
%\curraddr{}
\email{lescop@ujf-grenoble.fr}
\thanks{Institut Fourier, CNRS, UJF Grenoble}

%    The 2010 edition of the Mathematics Subject Classification is
%    now available.  If you are citing a classification from the
%    new scheme, use the following input coding instead.
%\subjclass[2010]{57M27 57N10 57M25 55R80}

\begin{abstract}
Let $M$ be a closed oriented $3$-manifold with first Betti number one.
Its equivariant linking pairing may be seen as a two-dimensional cohomology class in an appropriate infinite cyclic covering of the configuration space of ordered pairs of distinct points of $M$. We show how to define
the equivariant cube $\CQ(M,\KK)$ of this Blanchfield pairing with respect to a framed knot $\KK$ that generates $H_1(M;\ZZ)/\mbox{Torsion}$.

We present the invariant $\CQ(M,\KK)$ and some of its properties including a surgery formula.

Via surgery, the invariant $\CQ$ is equivalent to an invariant $\hat{\CQ}$ of null-homologous knots in rational homology spheres,
that is conjecturally equivalent to the two-loop part of the Kontsevich integral.

We generalize the construction of $\hat{\CQ}$
to obtain a topological construction for an invariant that is conjecturally equivalent to the whole Kricker rational lift of the Kontsevich integral for null-homologous knots in rational homology spheres.
\end{abstract}

\maketitle

\section{Introduction}
\label{secintro}

\subsection{Background}
The study of $3$--manifold invariants built from
integrals over configuration spaces started
after the work of Witten on Chern-Simons theory in 1989 \cite{wit},
with work of Axelrod, Singer~\cite{as1,as2}, Kontsevich~\cite{ko}, Bott, Cattaneo~\cite{BC,bc2,cat}, Taubes~\cite{taubes}.
In 1999, in \cite{kt}, G. Kuperberg and D. Thurston announced that
some of these invariants, the Kontsevich ones,
fit in with the framework of finite type invariants of homology spheres
studied by Ohtsuki, Le, J. and H. Murakami, Goussarov, Habiro, Rozansky, Garoufalidis, Polyak, Bar-Natan \cite{oht,ggp,lmo,hab,Aa1,Aa2,Aa3} and others.
They showed that these invariants together define a
universal finite type invariant for homology 3-spheres. 
I gave specifications on the Kuperberg-Thurston work in \cite{lesconst} and generalisations in \cite{sumgen}.

Similar studies for the knots and links cases had been performed by many other authors including Guadagnini, Martellini, Mintchev~\cite{gua}, Bar-Natan~\cite{bn'}, Kontsevich~\cite{Kon}, Polyak, Viro~\cite{pv}, Bott, Taubes~\cite{bt}, Altsch\"uler, Freidel~\cite{af}, D. Thurston~\cite{th}, Poirier~\cite{Po}. See also the Labastida survey \cite{La} and the references therein.

The above mentioned Kuperberg-Thurston work shows how to write the Casson invariant $\lambda$, originally defined by Casson in 1984 as an algebraic number
of conjugacy classes of irreducible $SU(2)$-representations \cite{akmc,gm,mar}, as
$$\lambda(N) = \frac{1}{6} \int_{\left(N\setminus \{\infty\}\right)^2\setminus \mbox{diagonal}} \omega^3$$
for a homology sphere $N$ (a closed oriented $3$-manifold with the same integral homology as $S^3$), a point $\infty$ in $N$,
and a closed $2$-form $\omega$ such that
for any 2-component link 
$$J \sqcup L \colon S^1 \sqcup S^1 \rightarrow N \setminus \{\infty\},$$the linking number of $J$ and $L$ reads
$$ lk(J,L)=\int_{J\times L}\omega.$$

In this sense, $6\lambda(N)$ may be viewed as the {\em cube of the linking form of $N$.} It can also be expressed as the  algebraic triple intersection $\langle F_X,F_Y,F_Z \rangle$
of three codimension $2$ cycles $F_X$, $F_Y$, $F_Z$ of $(C_2(N), \partial C_2(N))$ (Poincar\'e dual to the previous $\omega$)
for a compactification $C_2(N)$ of $\left(\left(N\setminus \{\infty\}\right)^2\setminus \mbox{diagonal}\right)$ that is a $6$--manifold with boundary.
Here, for any 2-component link $(J,L)$ of $(N\setminus \{\infty\})$
as above, the linking number of $J$ and $L$ is the algebraic intersection of $J\times L$ and $F_X$, (or $F_Y$ or $F_Z$) in the compactification $C_2(N)$.
A complete definition of $\lambda$ in these terms is described in the appendix.

\subsection{Introduction to the results}
\label{subint}

In the first part of this article, we shall present a similar construction for an
{\em equivariant cube ${\mathcal Q}(M,\KK)$ of the equivariant linking pairing} for a closed $3$--manifold $M$ with $H_1(M;\QQ)=\QQ$,
with respect to a framed knot $\KK=(K,K_{\parallel})$, that is a knot $K$ equipped with a parallel $K_{\parallel}$, such that $H_1(M;\ZZ)/\mbox{Torsion}=\ZZ [K]$.

Our invariant will live in the field of rational functions $\QQ(x,y)$.
The simplest example of a pair $(M,\KK)$ as above is the pair
$(S^1 \times S^2,S^1 \times u)$ where $S^1 \times u$ is equipped with a parallel. Note that the choice of the parallel does not affect the diffeomorphism class of the pair $(M,\KK)$ in this case. We shall have
$${\mathcal Q}(S^1 \times S^2,S^1 \times u)=0.$$
Furthermore, if $N$ is a rational homology sphere, and if $\sharp$ stands for the connected sum,
$${\mathcal Q}(M \sharp N,\KK)={\mathcal Q}(M,\KK) + 6 \lambda(N)$$
where $\lambda$ is the Walker generalization of the Casson invariant
normalized like the Casson invariant in \cite{akmc,gm,mar}. If $\lambda_W$ denotes the Walker invariant normalized as in \cite{wal}, then $\lambda=\frac{\lambda_W}{2}$.

We shall also state a surgery formula in Proposition~\ref{propsur} for our invariant, and we shall
determine the vector space spanned by the differences
$({\mathcal Q}(M,\KK^{\prime})-{\mathcal Q}(M,\KK))$ for other framed knots $\KK^{\prime}$ whose homology classes generate $H_1(M;\ZZ)/\mbox{Torsion}$, in Proposition~\ref{propfrakcha}.
This determination will allow us to define an induced invariant
for closed oriented $3$-manifolds with first Betti number one.
This latter invariant should be equivalent
to a special case (the two-loop case) of invariants combinatorially defined by Ohtsuki in 2008, in \cite{ohtb},
for $3$-manifolds of rank one.

Let $M_{\KK}$ be the manifold obtained from $M$ by surgery on $\KK$: This manifold is obtained from $M$ by replacing a tubular neighborhood of $K$ by another solid torus $N(\hat{K})$ whose meridian is the given parallel $K_{\parallel}$ of $K$. It is a  rational homology sphere and the core $\hat{K}$ of the new torus $N(\hat{K})$ is a null-homologous knot in $M_{\KK}$. 

Our data $(M,\KK)$ are equivalent to the data $(M_{\KK},\hat{K})$. Indeed, $M$ is obtained from $M_{\KK}$ by $0$-surgery on $\hat{K}$.
Hence our invariant can be seen as an invariant of null-homologous knots in rational homology spheres.
For these (and even for boundary links in rational homology spheres),
following conjectures of Rozansky \cite{Ro1}, Garoufalidis and Kricker defined a rational lift of the Kontsevich integral in \cite{Kr,GK}, that generalizes the Rozansky $2$--loop invariant of knots in $S^3$ of \cite[Section 6, 6.9]{Ro1}.
The two-loop part of this Kricker lift for knots is often called the two-loop polynomial. Its history and many of its properties
are described in \cite{oht2}.
Our invariant shares many features with this two-loop polynomial and is certainly {\em equivalent\/} to this invariant, in the sense that if one of the invariants distinguishes two knots with equivalent equivariant linking pairing, then the other one does. It could even be equal to the two-loop polynomial.

In 2005, Julien March\'e also proposed a similar ``cubic'' definition of an invariant equivalent to the two-loop polynomial \cite{Ju}.

In terms of Jacobi diagrams or Feynman graphs, the Casson invariant 
was associated with the graph $\theta$ and
our equivariant cube is associated with the graph $\theta$ with hair or beads.

All the results of the first part of this article are proved in \cite{betaone}. 

In the second part of this article, we explain how the topological construction of ${\mathcal Q}(M,\KK)$ generalizes to the construction of an invariant of $(M,\KK)$ that should be equivalent to the Kricker rational lift of the Kontsevich integral of null-homologous knots in rational homology spheres.

This article is an expansion of the talk I gave at the conference {\em Chern-Simons Gauge theory: 20 years after, Hausdorff center for Mathematics\/} in Bonn in August 2009. I thank the organizers Joergen Andersen, Hans Boden, Atle Hahn and Benjamin Himpel of this great conference. 

The first part of the article and the appendix are of expository nature and do not contain all the proofs; that first part may be considered as a research announcement for the results of \cite{betaone}. The second part relies on some results of the first part and contains the construction of a more powerful invariant of $(M,\KK)$ with the proof of its invariance.

I started to work on this project after a talk of Tomotada Ohtsuki for
a workshop at the CTQM in {\AA}rhus in Spring 2008. I thank Joergen Andersen and Bob Penner for organizing this very stimulating meeting, and Tomotada Ohtsuki for discussing this topic with me. Last but not least, I thank the referee for preventing the invariants constructed in the second part from living in a far less interesting space.

\subsection{Conventions}

All the manifolds considered in this article are oriented.
Boundaries are oriented by the {\em outward normal first\/} convention. The fiber $N_u(A)$ of the normal bundle $N(A)$ of a submanifold $A$ in a manifold $C$ at $u \in A$ is oriented so that $T_uC=N_u(A)\oplus T_uA$ as oriented vector spaces.
For two transverse submanifolds $A$ and $B$ of $C$, $A \cap B$ is oriented so that $N_u(A \cap B)=N_u(A)\oplus N_u(B)$. 
When the sum of the dimensions of $A$ and $B$ is the dimension of $C$, and when $A \cap B$ is finite, the {\em algebraic intersection\/} $\langle A, B \rangle$ of $A$ and $B$ in $C$ is the sum of the signs of the points of $A\cap B$, where the sign of an intersection point $u$ of $A \cap B$ is $1$ if and only if
$T_uC=N_u(A)\oplus N_u(B)$ (that is if and only if $T_uC=T_uA\oplus T_uB$) as oriented vector spaces. It is $(-1)$ otherwise. The algebraic intersection of $n$ compact transverse submanifolds  $A_{1}$,  $A_{2}$ ...,  $A_{n}$ of $C$ whose codimensions sum is the dimension of $C$ is defined similarly. The sign of an intersection point $u$ is $1$ if and only if
$T_uC=N_u(A_1)\oplus N_u(A_2)\oplus \dots \oplus N_u(A_n)$ as oriented vector spaces.

\subsection{On the equivariant linking pairing}

Fix $(M,\KK)$ as in Subsection~\ref{subint}.
Let $$\projM \colon \tilde{M} \rightarrow M$$ denote the regular infinite cyclic covering of $M$, and let $\theta_M$ be the generator of its covering group that corresponds to the action of the class of $K$.
The action of $\theta_{M}$ on $H_1(\tilde{M};\QQ)$ is denoted as the multiplication by $t_{M}$.\\
The $\QQ[t_{M}^{\pm 1}]$-module $H_1(\tilde{M};\QQ)$ reads
$$H_1(\tilde{M};\QQ)=\bigoplus_{i=1}^k \frac{\QQ[t_{M}^{\pm 1}]}{\delta_i}$$
for polynomials $\delta_i$ of $\QQ[t_{M}^{\pm 1}]$ where $\delta_i$ divides $\delta_{i+1}$.
Then
$\delta=\delta(M)=\delta_k$ is the
{\em annihilator} of $H_1(\tilde{M};\QQ)$ and
$\Delta=\Delta(M)=\prod_{i=1}^k\delta_i$ is the {\em Alexander polynomial\/} of $M$.\\
These very classical invariants are normalised so that  $\Delta(t_M)=\Delta(t_M^{-1})$, $\Delta(1)=1$,
 $\delta(t_M)=\delta(t_M^{-1})$ and $\delta(1)=1$. (In order to make $\delta$ symmetric, we may have to allow it to belong to $\left((t_M^{1/2}+t_M^{-1/2})\QQ[t_{M}^{\pm 1}] \cup \QQ[t_{M}^{\pm 1}]\right)$.)
Note that $\Delta$ and $\delta$ coincide when $\Delta$ has no multiple roots.

Let $(J,L)$ be a two-component link of $\tilde{M}$ such that
$\projM(J) \cap \projM(L)=\emptyset$. If $J$ bounds a (compact) surface $\Sigma$
in $\tilde{M}$ transverse to $\theta_M^{n}(L)$ for all $n$, define the {\em equivariant intersection\/} $\langle \Sigma, L \rangle_e$ as
$$\langle \Sigma, L \rangle_e = \sum_{n \in \ZZ} t_M^n \langle \Sigma, \theta_M^{n}(L) \rangle$$
where $\langle \Sigma, \theta_M^{n}(L) \rangle$ is the usual algebraic intersection.
Then the {\em equivariant linking pairing\/} of $J$ and $L$ is 
$$lk_e(J,L)= \langle \Sigma, L \rangle_e.$$
In general, $\delta(\theta_M)(J)$ bounds a surface $\delta(\theta_M)\Sigma$ and
$$lk_e(J,L)= \frac{\langle \delta(\theta_M)\Sigma, L \rangle_e}{\delta(t_M)}.$$
For any two one-variable polynomials $P$ and $Q$,
$$lk_e(P(\theta_M)(J),Q(\theta_M)(L))=P(t_M)Q(t_M^{-1})lk_e(J,L).$$

\subsection{The construction of $\CQ(M,\KK)$}
\label{subconst}

Consider the infinite cyclic covering of $M^2$
 $$\widetilde{M^2} = \frac{\tilde{M}^2}{(u,v)\sim (\theta_M(u),\theta_M(v))} \;\;\;\;\hfl{\proj} \;\;\;\; M^2$$
 with generating covering transformation $\theta$.
$$\theta(\overline{(u,v)})=\overline{(\theta_M(u),v)}=\overline{(u,\theta_M^{-1}(v))}.$$
The diagonal of $\left(\tilde{M}^2 \right)$ projects to a preferred lift of the diagonal of $M^2$ in $\TCMD$.
$$\proj^{-1}(\mbox{diag}(M^2))=\sqcup_{n\in \ZZ} \theta^n\left(\overline{\mbox{diag}(\tilde{M}^2)}=\mbox{diag}(M^2)\right)=\ZZ \times \mbox{diag}(M^2).$$

\begin{definition}
\label{defblowup}
In a smooth $n$-manifold $C$, a tubular neighborhood of a compact $m$-submanifold $A$ locally reads as $\RR^{n-m} \times U$ for some open $U \subset A$, where $\RR^{n-m}=(]0,\infty[ \times S^{n-m-1}) \cup \{0\}$ stands for the fiber of the normal bundle of $A$, and $S^{n-m-1}$ stands for the fiber of the unit normal bundle of $A$.
In this article, the manifold $C(A)$ obtained by {\em blowing up\/} $A$ in $C$ is obtained by replacing $A$ by its unit normal bundle in $A$.
Near $U$, $\RR^{n-m} \times U$ is replaced by $[0,\infty[ \times S^{n-m-1} \times U$. The blown-up manifold $C(A)$ is homeomorphic to
the complement of an open tubular neighborhood of $A$ in $C$, but it has a canonical smooth projection onto $C$ and a canonical smooth structure. When $C$ and $A$ are compact, $C(A)$ is a compactification of $C \setminus A$.
\end{definition}

The {\em configuration space\/} ${C}_2(M)$ is obtained from $M^2$ by blowing up the diagonal of $M^2$ in this sense, and
the configuration space $\tilde{C}_2(M)$ is obtained from $\widetilde{M^2}$ by blowing up $\proj^{-1}(\mbox{diag}(M^2))$. The transformation $\theta$ of $\widetilde{M^2}$ naturally lifts to a transformation of $\tilde{C}_2(M)$ that is still denoted by $\theta$.
The quotient of $\tilde{C}_2(M)$ by the action of $\theta$ is ${C}_2(M)$.

Since the normal bundle of the diagonal of $M^2$ is canonically equivalent to the tangent bundle $TM$ of $M$ via $$(u,v) \in \frac{TM^2}{\mbox{diag}(TM^2)} \mapsto (v-u) \in TM,$$
the unit normal bundle of $\proj^{-1}(\mbox{diag}(M^2))$ is $\ZZ \times ST(M)$ where $ST(M)$ is the unit tangent bundle of $M$ so that
$$\partial \tilde{C}_2(M)=\ZZ \times ST(M).$$
A trivialisation $\tau: TM \rightarrow M \times \RR^3$ of $TM$ identifies $ST(M)$ with $M \times S^2$.

Like any oriented closed $3$-manifold, $M$ bounds an (oriented compact) manifold $W^4$ with signature $0$. Then $TW^4_{|M}=\RR \oplus TM$.
A trivialisation $\tau$ of $TM$ induces a trivialisation of $TW^4 \otimes \CC$ on $M$. 
The {\em first Pontrjagin class\/} $p_1(\tau)$ of such a trivialisation $\tau$ of the tangent bundle of $M$ is the obstruction $p_1(W^4;\tau)$ to extend this trivialisation to $W^4$. It belongs to $H^4(W^4,M;\pi_3(SU(4)))= \ZZ$. We use the notation and conventions of \cite{milnorsta}, see also \cite[Section 1.5]{lesconst}.

Now, the construction of $\CQ(M,\KK)$ is given by the following theorem.
\begin{theorem}
\label{thminvtripl}
Let $\tau: TM \rightarrow M \times \RR^3$ be a trivialisation of $TM$
and let $p_1(\tau)$ be its first Pontrjagin class.
Assume that $\tau$ maps the oriented unit tangent vectors of $K$ to some fixed
$\qvarM \in S^2$. Then $\tau$ induces a parallelisation of $K$. Let $K_{\fvarM}$, $K_{\svarM}$, $K_{\tvarM}$ be three disjoint parallels of $K$, on the boundary $\partial N(K)$ of a tubular neighborhood of $K$, that induce the same parallelisation of $K$ as $\tau$.

Consider the continuous map
$$\begin{array}{llll}\check{A}(K) \colon &(S^1=[0,1]/(0\sim 1)) \times [0,1] &\rightarrow & C_2(M)\\
&(t,u \in]0,1[)& \mapsto &(K(t),K(t+u)),
\end{array}$$
and its lift ${A}(K) \colon S^1 \times [0,1] \rightarrow  \TCM$
such that the lift of $(K(t),K(t+\varepsilon))$ is in a small neighborhood of the canonical lift of the diagonal, for a small positive $\varepsilon$. Let $A(K)$ also denote the $2$--chain $A(K)(S^1 \times [0,1])$.

For $\cvarM \in S^2$,
let $$s_{\tau}(M;\cvarM)=\tau^{-1}(M \times \cvarM) \subset \left(ST(M)= \{0\} \times ST(M)\right) \subset \partial \TCM.$$

Let
 $$I_{\Delta}(t)=\frac{1+t}{1-t} +  \frac{t\Delta^{\prime}(t)}{\Delta(t)}$$
where $\Delta=\Delta(M)$.
Let $\fvarM$, $\svarM$, $\tvarM$ be three distinct points in $S^2 \setminus\{\qvarM, -\qvarM\}$.

There exist three rational transverse\/ $4$--dimensional chains $G_{\fvarM}$, $G_{\svarM}$ and $G_{\tvarM}$ of $\TCM$ whose boundaries
are 
$$\partial G_{\fvarM} =(\theta -1)\delta(\theta)\left(s_{\tau}(M;\fvarM)-\ID(\theta) ST(M)_{|K_{\fvarM}}\right) ,$$ 
$$\partial G_{\svarM} =(\theta -1)\delta(\theta)\left(s_{\tau}(M;\svarM)-\ID(\theta) ST(M)_{|K_{\svarM}}\right) \mbox{and}$$ $$\partial G_{\tvarM} =(\theta -1)\delta(\theta)\left(s_{\tau}(M;\tvarM)-\ID(\theta) ST(M)_{|K_{\tvarM}}\right)$$
and such that the following equivariant algebraic intersections in $\TCM$ vanish
$$\langle G_{\fvarM}, A(K)\rangle_e=\langle G_{\svarM}, A(K)\rangle_e=\langle G_{\tvarM}, A(K)\rangle_e=0.$$
Define the equivariant algebraic triple intersection in $\TCM$
$$\langle G_{\fvarM},G_{\svarM},G_{\tvarM} \rangle_e=\sum_{(i,j)\in \ZZ^2} \langle G_{\fvarM},\theta^{-i}(G_{\svarM}),\theta^{-j}G_{\tvarM})\rangle_{\TCM} \svar^i\tvar^j \in \QQ[\svar^{\pm 1},\tvar^{\pm 1}].$$
 Let $R_{\delta}= \frac{\QQ[\fvar^{\pm 1},\svar^{\pm 1}, \tvar^{\pm 1},\frac{1}{\delta(\fvar)},\frac{1}{\delta(\svar)},\frac{1}{\delta(\tvar)}]}{(\fvar \svar \tvar=1)}.$
Then $$\CQ(M,\KK)=\frac{\langle G_{\fvarM},G_{\svarM},G_{\tvarM} \rangle_e}{(\fvar -1)(\svar -1)(\tvar -1)\delta(\fvar)\delta(\svar)\delta(\tvar)} - \frac{p_1(\tau)}{4} \in R_{\delta}$$ and $\CQ(M,\KK)$ only depends on the isotopy class of the knot $K$ and on its parallelisation.
Furthermore,
%$$\delta(\fvar)\delta(\svar)\delta(\tvar)\CQ(M,\KK) \in \frac{\QQ[\fvar^{\pm 1},\svar^{\pm 1}, \tvar^{\pm 1}]}{(\fvar \svar \tvar=1)},$$ 
$$\CQ(M,\KK)(\fvar,\svar,\tvar)=\CQ(M,\KK)(\svar,\fvar,\tvar)=\CQ(M,\KK)(\tvar,\svar,\fvar)=\CQ(M,\KK)(\fvar^{-1},\svar^{-1},\tvar^{-1})$$ and $\CQ(M,\KK)$ does not depend on the orientation of $K$.
\end{theorem}

Of course, the theorem above contains a lot of statements.
Let us explain their flavour.
Consider the homology of $\TCM$ with coefficients in $\QQ$
endowed with the structure of $\QQ[t,t^{-1}]$-module where the multiplication by $t$ is induced by the action of $\theta$ on $\TCM$. 
Let $\QQ(t)$ be the field of fractions of $\QQ[t,t^{-1}]$ and set
$$H_{\ast}(C_2(M);\QQ(t)) = H_{\ast}(\TCM;\QQ) \otimes_{\QQ[t,t^{-1}]} \QQ(t).$$

\begin{lemma}
\label{lemhomTCM}
$H_{i}(C_2(M);\QQ(t))\cong H_{i-2}(M;\QQ) \otimes_{\QQ} \QQ(t)$
for any $i \in \ZZ$.
$$H_{2}(C_2(M);\QQ(t))=\QQ(t)[ST(M)_{|\ast} (\cong \ast \times S^2)]$$
$$H_{3}(C_2(M);\QQ(t))=\QQ(t)[ST(M)_{|K} (\cong K \times S^2)]$$
$$H_{4}(C_2(M);\QQ(t))=\QQ(t)[ST(M)_{|S} (\cong S \times S^2)]$$
where $S$ is a closed surface of $M$ such that $H_2(M;\ZZ)=\ZZ[S]$.
%$$H_{5}(C_2(M);\QQ(t))=\QQ(t)[ST(M) (\cong M \times S^2)].$$
\end{lemma}

In particular, the statement of the theorem contains the following lemma:

\begin{lemma}
 The homology class in $H_{3}(C_2(M);\QQ(t))$ of a global section of $ST(M)$ induced by a trivialisation of $M$ is $\ID [ST(M)_{|K}]$.
\end{lemma}
Therefore, $\ID [ST(M)_{|K}]$ cannot be removed from our boundaries.
Lemma~\ref{lemhomTCM} can be proved by classical means. 
Observe that since $$\langle A(K), ST(M)_{|S} \rangle_e =1-t^{-1},$$
the class of $A(K)$ in $H_{2}(\TCM,\partial \TCM)$ detects $[ST(M)_{|S}]$. Thus, the condition $\langle A(K), G_{\fvarM} \rangle_e =0$ ensures that if $G^{\prime}_{\fvarM}$ satisfies the same conditions as $G_{\fvarM}$, $(G^{\prime}_{\fvarM}-G_{\fvarM})$ bounds a $5$-chain and 
$\langle G^{\prime}_{\fvarM},G_{\svarM},G_{\tvarM} \rangle_e=\langle G_{\fvarM},G_{\svarM},G_{\tvarM} \rangle_e.$
Therefore, our algebraic intersection $\langle G_{\fvarM},G_{\svarM},G_{\tvarM} \rangle_e$ is well-defined.

The class of $F_{\fvarM}=\frac{G_{\fvarM}}{(t -1)\delta(t)}$
(that is the same as the class of $F_{\svarM}=\frac{G_{\svarM}}{(t -1)\delta(t)}$ or $F_{\tvarM}=\frac{G_{\tvarM}}{(t -1)\delta(t)}$) in $H_{4}(C_2(M),\partial C_2(M);\QQ(t))$ is dual to $[ST(M)_{|\ast}]$:
$$\langle ST(M)_{|\ast}, F_{\fvarM} \rangle_e =1.$$

For a two-component link $(J,L)$ of $\tilde{M}$ such that
$\projM(J) \cap \projM(L)=\emptyset$, the class of $J \times L$ in $\TCM$ reads
$lk_e(J,L) [ST(M)_{|\ast}]$. By the above equation, this can be rewritten as 
$$lk_e(J,L)=\langle J \times L, F_{\fvarM} \rangle_e$$
and the chains $F_{\cvarM}$ represent the equivariant linking number in this sense.

Recall that all the assertions of this section are proved in details in \cite{betaone}.

\subsection{A few properties of $\CQ$}

Recall that
$\lambda$ denotes the Casson-Walker invariant normalised like the Casson invariant. I added the following proposition in order to answer a question that George Thompson asked me at the conference {\em Chern-Simons Gauge theory : 20 years after\/} in Bonn.
I thank him for asking.
%: if $\lambda_W$ denotes the Walker invariant normalized as in \cite{wal}, then $\lambda=\frac{\lambda_W}{2}$.

\begin{proposition}
 Let $M_{\KK}$ denote the rational homology sphere obtained from $M$ by surgery along
$\KK$. Then $$\CQ(M,\KK)(1,1,1)=6\lambda(M_{\KK}).$$
\end{proposition}

For a function $f$ of $\fvar$, $\svar$ and $\tvar$, $\sum_{\mathfrak{S}_3(\fvar,\svar,\tvar)}f(\fvar,\svar,\tvar)$ stands for $$\sum_{\sigma \in \mathfrak{S}_3(\fvar,\svar,\tvar)}f(\sigma(\fvar),\sigma(\svar),\sigma(\tvar))$$
where $\mathfrak{S}_3(\fvar,\svar,\tvar)$ is the set of permutations of $\{\fvar,\svar,\tvar\}$.

\begin{proposition}
\label{propsur}
Let $J$ be a knot of $M$ that bounds
a Seifert surface $\Sigma$ disjoint from $K$
whose $H_1$ goes to $0$ in $H_1(M)/\mbox{Torsion}$.\\
Let $p/q$ be a nonzero rational number.
Let $(a_i,b_i)_{i=1,\dots,g}$ be a symplectic basis of $H_1(\Sigma)$.
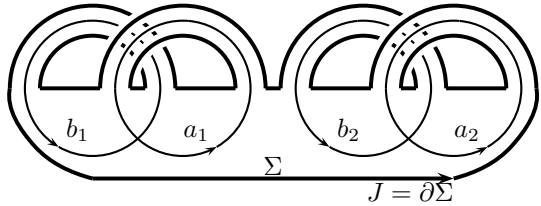
\begin{center}
\begin{pspicture}[shift=-0.1](0,.2)(5.4,2.9)
\psarc[linewidth=1.5pt](1.1,1.6){1.1}{0}{180}
\psarc[linewidth=1.5pt](1.1,1.6){.7}{0}{180}
\psarc{->}(1.1,1.6){.9}{95}{240}
\psarc(1.1,1.6){.9}{240}{95}
\psarc[linewidth=1.5pt,border=2pt](2.3,1.6){1.1}{0}{180}
\psarc[linewidth=1.5pt,border=2pt](2.3,1.6){.7}{0}{180}
\psarc[border=1pt](2.3,1.6){.9}{85}{180}
\psarc{->}(2.3,1.6){.9}{180}{300}
\psarc(2.3,1.6){.9}{-60}{85}
\psecurve[linewidth=1.5pt](1.1,2.7)(0,1.6)(.33,.83)(1.1,.4)(1.3,.4)
\psline[linewidth=1.5pt]{->}(1.1,.4)(5.9,.4)
\rput[b](3.5,.45){$\Sigma$}
\rput[rt](5.9,.35){$J=\partial \Sigma$}
\rput[rb](2.65,.9){$a_1$} 
\rput[lb](.75,.9){$b_1$} 
\psline[linewidth=1.5pt](.4,1.6)(1.2,1.6)
\psline[linewidth=1.5pt](1.6,1.6)(1.8,1.6)
\psline[linewidth=1.5pt](2.2,1.6)(3,1.6)
\psarc[linewidth=1.5pt](4.7,1.6){1.1}{0}{180}
\psarc[linewidth=1.5pt](4.7,1.6){.7}{0}{180}
\psarc{->}(4.7,1.6){.9}{95}{240}
\psarc(4.7,1.6){.9}{240}{95}
\psarc[linewidth=1.5pt,border=2pt](5.9,1.6){1.1}{0}{180}
\psarc[linewidth=1.5pt,border=2pt](5.9,1.6){.7}{0}{180}
\psarc[border=1pt](5.9,1.6){.9}{85}{180}
\psarc{->}(5.9,1.6){.9}{180}{300}
\psarc(5.9,1.6){.9}{-60}{40}
\psarc[border=1pt](5.9,1.6){.9}{40}{85}
\psecurve[linewidth=1.5pt](5.9,2.7)(7,1.6)(6.67,.83)(5.9,.4)(5.7,.4)
\rput[rb](6.25,.9){$a_2$} 
\rput[lb](4.35,.9){$b_2$} 
\psline[linewidth=1.5pt](3.4,1.6)(3.6,1.6)
\psline[linewidth=1.5pt](4,1.6)(4.8,1.6)
\psline[linewidth=1.5pt](5.2,1.6)(5.4,1.6)
\psline[linewidth=1.5pt](5.8,1.6)(6.6,1.6)
\end{pspicture}
\end{center}
Let $$\lambda_e^{\prime}(J)=\frac{1}{12}\sum_{(i,j) \in \{1,\dots,g\}^2}\sum_{\mathfrak{S}_3(\fvar,\svar,\tvar)} \left(\alpha_{ij}(x,y)+\alpha_{ij}(x^{-1},y^{-1}) +\beta_{ij}(\fvar,\svar)\right) \in R_{\delta}$$ where 
$$\alpha_{ij}(x,y)=lk_e(a_i ,a_j^+)(x)lk_e(b_i ,b_j^+)(y) -lk_e(a_i ,b_j^+)(x)lk_e(b_i ,a_j^+)(y)$$ and
$$ \beta_{ij}(\fvar,\svar)=\left(lk_e(a_i,b^+_i )(\fvar)-lk_e(b^+_i,a_i)(\fvar)\right)\left(lk_e(a_j,b^+_j )(\svar)-lk_e(b^+_j,a_j)(\svar)\right),$$
then $$\CQ(M(J;p/q),\KK)-\CQ(M,\KK)= 6\frac{q}{p}\lambda_e^{\prime}(J) + 6 \lambda(S^3(U ;p/q))$$
where $S^3(U ;p/q)$ is the lens space $L(p,-q)$ obtained from $S^3$ by $p/q$--surgery on the unknot $U$.
\end{proposition}

Since $H_1(\Sigma)$ goes to $0$ in $H_1(M)/\mbox{Torsion}$
in the above statement, $\Sigma$ lifts as homeomorphic copies of $\Sigma$ and $lk_e(a_i ,a_j^+)$ denotes the equivariant linking number of a lift of $a_i$ in $\tilde{M}$ in some lift of $\Sigma$ and a lift of $a^+_j$ near the same lift of $\Sigma$. The superscript $+$ means that 
$a_j$ is pushed in the direction of the positive normal to $\Sigma$.

In \cite{betaone}, we deduce the surgery formula of Proposition~\ref{propsur} for surgeries on knots from a surgery formula for Lagrangian-preserving replacements of rational homology handlebodies.

When the above knot $J$ is inside a rational homology ball, $\lambda_e^{\prime}(J)$ coincides with $\frac{1}{2}\Delta^{\prime\prime}(J)(1)$, where $\Delta(J)$ is the Alexander polynomial of $J$, and the right-hand side is nothing but $6$ times the variation of the Casson-Walker invariant 
under a $p/q$--surgery on $J$. Since any rational homology sphere can be obtained from $S^3$ by a sequence of surgeries on null-homologous knots in rational homology spheres with nonzero coefficients, after a possible connected sum with lens spaces, we easily deduce the following proposition from the above surgery formula.

\begin{proposition}
\label{propconcas}
 Let $N$ be a rational homology sphere, then 
$$\CQ(M \sharp N, \KK)=\CQ(M, \KK) + 6 \lambda(N).$$
\end{proposition}

Recall $I_{\Delta}(t)=\frac{1+t}{1-t} +  \frac{t\Delta^{\prime}(t)}{\Delta(t)}$.

\begin{proposition}
\label{propfrakcha}
Let $\KK^{\prime}$ be another framed knot of $M$ such that $$H_1(M)/\mbox{Torsion} = \ZZ[K^{\prime}].$$
Then there exists an antisymmetric polynomial 
${\mathcal V}(\KK,\KK^{\prime})$ in $\QQ[t,t^{-1}]$ such that 
$$\CQ(M,\KK^{\prime}) - \CQ(M,\KK)=\sum_{\mathfrak{S}_3(x,y,z)}\frac{{\mathcal V}(\KK,\KK^{\prime})(x)}{\delta(x)}I_{\Delta}(y).$$

Furthermore, for any $k \in \ZZ$, there exists a pair of framed knots $(\KK,\KK^{\prime})$ such that ${\mathcal V}(\KK,\KK^{\prime})=q(t^k-t^{-k})$ for some nonzero rational number $q$.
\end{proposition}

\begin{proposition} If $\KK=(K,K_{\parallel})$ and if $\KK^{\prime}=(K,K^{\prime}_{\parallel})$, where $K^{\prime}_{\parallel}$ is another parallel of $K$ such that the difference $(K^{\prime}_{\parallel} -K_{\parallel})$ is homologous to a positive meridian of $K$ in $\partial N(K)$, then
$${\mathcal V}(\KK,\KK^{\prime})(t)=-\frac{\delta (t)}{2}\frac{t\Delta^{\prime}(t)}{\Delta(t)}.$$
\end{proposition}

\begin{proposition}
\label{propcorcalvarb} 
If $K$ and $K^{\prime}$ coincide along an interval, if $K^{\prime}-K$ bounds a surface $\Bor$ that lifts in $\tilde{M}$ such that $(K^{\prime}_{\parallel} - K_{\parallel})$ is homologous to a curve of $\Bor$ in the complement of $(\partial \Bor \cup K)$ in a regular neighborhood of $\Bor$, and
if $(a_i,b_i)_{i \in \{1,\dots,g\}}$ is a symplectic basis of $H_1(\Bor;\ZZ)$, then
$$\frac{{\mathcal V}(\KK,\KK^{\prime})(t)}{\delta (t)}=\sum_{i=1}^g\left(lk_e(a_i,b_i^+) -{lk_e(b_i^+,a_i)}\right).$$
\end{proposition}

\subsection{The derived $3$-manifold invariant}

\begin{definition}{\em Definition of an invariant for $3$-manifolds of rank one:\/}\\
Let $Q_k(\delta,\Delta)=\sum_{\mathfrak{S}_3(x,y,z)}\frac{x^k-x^{-k}}{\delta(x)}I_{\Delta}(y)$ for $k \in (\NN \setminus \{0\})$.
For a fixed $(\delta,\Delta)$, define $\overline{\CQ}(M)$ as the class of $\CQ(M,\KK)$ in the 
quotient of $R_{\delta}$ by the vector space generated by the $Q_k(\delta,\Delta)$ for $k \in (\NN \setminus \{0\})$.
\end{definition}

Thanks to Proposition~\ref{propfrakcha}, $\overline{\CQ}(M)$ is an invariant of $M$. This invariant is certainly equivalent
to a special case (the two-loop case) of invariants combinatorially defined by Ohtsuki in 2008 in \cite{ohtb},
for $3$-manifolds of rank one,  when $\delta=\Delta$.
The following proposition shows that it often detects the connected sums with rational homology spheres with non-trivial Casson-Walker invariants.

\begin{proposition}
If $\Delta$ has only simple roots and
if $N$ is a rational homology sphere such that $\lambda(N)\neq 0$,
then $\overline{\CQ}(M) \neq \overline{\CQ}(M\sharp N)$.
\end{proposition}

\subsection{An alternative definition for $\CQ$ with boundary conditions}

\label{subdefbord}

Write the sphere $S^2$ as the quotient of $[0,8] \times S^1$ where $\{0\}\times S^1$ is identified with a single point (the North Pole of $S^2$) and $\{8\}\times S^1$ is identified with another single point (the South Pole of $S^2$).
When $\alpha \subset [0,8]$, $D^2_{\alpha}$ denotes the image of $\alpha \times S^1$ via the quotient map $\projq$. For example, $D^2_{[1,8]}$ is a disk. Embed $D^2_{[1,8]} \times S^1$ as a tubular
neighborhood of $K$, so that $K=\{\ast_{\qvarM}\}\times S^1$ for some $\ast_{\qvarM} \in \partial D^2_{[0,5]}$, and $K_{\parallel}=\{q_{\fvarM}\}\times S^1$ for some $q_{\fvarM} \in \partial D^2_{[0,1]}$, and let $M_{[0,1]}=M\setminus (D^2_{]1,8]} \times S^1)$.
More generally, let 
$$\begin{array}{llll}
   r \colon & M & \rightarrow & [1,8]\\
  & x \in M_{[0,1]} & \mapsto & 1\\
& (\projq(t,z_{\qvarM}),z) \in D^2_{[1,8]} \times S^1&\mapsto &t.\\
\end{array}
$$
When $\alpha \subset [0,8]$, $M_{\alpha}=r^{-1}(\alpha)$
and $\tilde{M}_{\alpha}=\projM^{-1}(M_{\alpha})$.

Consider a map $\funcM \colon M \rightarrow S^1$ that coincides with the projection
onto $S^1$ on $D^2_{[1,8]} \times S^1$, and a lift of this map $\tilde{f}_M \colon \tilde{M} \rightarrow \RR$.
Embed $\tilde{M}_{[1,6]}=\projM^{-1}(D^2_{[1,6]} \times S^1)$ in $\RR^3$, seen as $\CC \times \RR$, 
as $\{z \in \CC; 1\leq |z| \leq 6\} \times \RR$ naturally so that the projection on $\RR$ is $\tilde{f}_M$. Here $\CC$ is thought of as horizontal and $\RR$ is vertical.
This embedding induces a trivialisation $\tau$ on $TM_{[1,6]}$ that we extend on $TM_{[0,6]}$.
This trivialisation respects the product structure with
$\RR$ on $\tilde{M}_{[1,6]}$, we also extend it on $\tilde{M}_{[1,8]}$
so that it still respects the product structure with
$\RR$, there.

\noindent{\em Construction of a map $\pi \colon \proj^{-1}\left((M_{[0,4]}^2 \setminus M_{[0,2[}^2)\setminus \mbox{\rm diag}({M}_{[2,4]}^2) \right)\rightarrow S^2$}\\
Fix $\varepsilon \in ]0,1/2]$. 
Let $$\begin{array}{llll}\chi \colon &[-4,4]& \rightarrow& [0,1]\\
&t \in [-\varepsilon,4] &\mapsto &1\\
&t \in [-4,-2\varepsilon] &\mapsto &0
\end{array}$$
be a smooth map. Recall that $\tilde{M}_{[1,4]}$ is embedded in $\RR^3$ that is seen as $\CC \times \RR$.
When $(u,v) \in \left(\tilde{M}_{[0,4]}^2 \setminus \tilde{M}_{[0,2[}^2\right) \setminus \mbox{diag}(\tilde{M}_{[2,4]}^2)$, set
$$U(u,v)=(1-\chi(r(u) -r(v)))(0,\tilde{f}_M(u)) + \chi(r(u)-r(v))u$$
$$V(u,v)=(1-\chi(r(v) -r(u)))(0,\tilde{f}_M(v)) + \chi(r(v)-r(u))v$$
so that $(U(u,v),V(u,v)) \in (\RR^3)^2 \setminus \mbox{diag}$.

Define $$\begin{array}{llll}\pi \colon &\proj^{-1}\left((M_{[0,4]}^2 \setminus M_{[0,2[}^2)\setminus \mbox{diag}({M}_{[0,4]}^2) \right)&\rightarrow& S^2\\
 &\overline{(u,v)}&\mapsto & \frac{V(u,v)-U(u,v)}{\parallel V(u,v)-U(u,v)\parallel}.
         \end{array}$$
When $A$ is a submanifold of $M$, $C_2(A)$ denotes the preimage of $A^2$ under the blowup map
from $C_2(M)$ to $M^2$, and $\tilde{C}_2(A)$ is the preimage of $C_2(A)$ under the covering map of $\TCM$.
The map $\pi$ naturally extends to $\tilde{C}_2(M_{[0,4]})\setminus \tilde{C}_2(M_{[0,2[})$.

Set $$J_{\Delta}=J_{\Delta}(t)=\frac{t\Delta^{\prime}(t)}{\Delta(t)}.$$
Let $S^2_H$ denote the subset of $S^2$ made of the vectors whose vertical coordinate is in $]0,\frac{1}{50}[$.
The following proposition is proved in \cite[Section 12]{betaone}.

\begin{proposition}
 \label{propdefbord}
Let $q_{\fvarM}$, $q_{\svarM}$ and $q_{\tvarM}$ be three distinct points
on $\partial D^2_{[0,1]}$, and let $\fvarM$, $\svarM$ and $\tvarM$ be three distinct vectors
of $S^2_H$.
For $\cvarM=\fvarM, \svarM$ or $\tvarM$, let $K_{\cvarM}=q_{\cvarM}\times S^1 $,
then
there exists a $4$-dimensional rational chain $C_{\cvarM}$ of $\tilde{C}_2(M_{[0,2]})$ whose boundary is
$$\delta(M)\left(\pi_{|\partial \tilde{C}_2(M_{[0,2]}) \setminus \partial \tilde{C}_2(M_{[0,2[})}^{-1}(\cvarM) \cup s_{\tau}(M_{[0,2]};\cvarM) \cup (-J_{\Delta}) ST(M)_{|K_{\cvarM}}\right),$$
and that is transverse
to $\partial \tilde{C}_2(M_{[0,2]})$, and $$\delta(M)(\fvar)\delta(M)(\svar)\delta(M)(\tvar)\CQ(M,\KK)=\langle C_{\fvarM}, C_{\svarM} , C_{\tvarM} \rangle_{e,\tilde{C}_2(M_{[0,2]})}-\frac{p_1(\tau)}{4}.$$
\end{proposition}
Note that when $M=S^1 \times S^2$, $J_{\Delta}=0$, $\pi$ extends 
to $\tilde{C}_2(M_{[0,2]})$ and $C_{\cvarM}=\pi_{| \tilde{C}_2(M_{[0,2]})}^{-1}(\cvarM)$ fulfills the conditions.
Also note that the proposition implies that $$\delta(M)(\fvar)\delta(M)(\svar)\delta(M)(\tvar)\CQ(M,\KK) \in \frac{\QQ[\fvar^{\pm 1},\svar^{\pm 1}, \tvar^{\pm 1}]}{(\fvar \svar \tvar=1)}.$$

\section{Construction of more general invariants}

The invariant that has been discussed so far corresponds to the
graph $\theta$ (with hair or beads), where the two vertices of the graph $\theta$ correspond to the two points of a configuration in $C_2(M)$, and the three edges are equipped with $\frac{1}{\delta}C_{\fvarM}$, $\frac{1}{\delta}C_{\svarM}$, and $\frac{1}{\delta}C_{\tvarM}$, respectively.

 The chains $C_{\cvarM}$ that were used in the definition of $\CQ$ can be used to define invariants $\tilde{z}_n(M,\KK)$ of $(M,\KK)$ associated to beaded
trivalent graphs with $2n$ vertices, and thus to configuration spaces
$C_{2n}(M)$ of $2n$ points.
We present the construction of these invariants below. Together, they will form a series $(\tilde{z}_n(M,\KK))_{n\in \NN}$, that should be equivalent to the  Kricker rational lift of the Kontsevich integral for null-homologous knots in rational homology spheres, described in \cite{Kr,GK}.

\subsection{On the target spaces}
\label{subdiagbead}

Here, a {\em trivalent graph\/} is a finite trivalent graph.
Such a graph will be said to be {\em oriented\/} if each of its vertices is equipped with a {\em vertex orientation\/}, that is a cyclic order of the three half-edges that meet at this vertex.
Such an oriented graph will be represented by the image of one of its planar immersions so that the vertex orientation is induced by the counterclockwise order of the half-edges meeting at this vertex.
$$
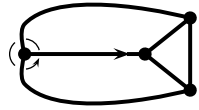
\begin{pspicture}[shift=-0.2](0,-.8)(2.8,.8)
\psset{xunit=.4cm,yunit=.4cm}
\psarc[linewidth=.5pt](1,0){.2}{15}{85}
\psarc[linewidth=.5pt](1,0){.2}{130}{230}
\psarc[linewidth=.5pt]{->}(1,0){.2}{275}{345}
\psline[linewidth=1.5pt]{-}(5,0)(4.4,0)
\psline[linewidth=1.5pt]{*->}(1,0)(4.5,0)
\psline[linewidth=1.5pt]{*-*}(6.5,-1.2)(5,0)
\psline[linewidth=1.5pt]{*-}(6.5,1.2)(5,0)
\psline[linewidth=1.5pt]{-}(6.5,1.2)(6.5,-1.2)
\pscurve[linewidth=1.5pt]{-}(1,0)(1,-1)(6.5,-1.2)
\pscurve[linewidth=1.5pt]{-}(1,0)(1,1)(6.5,1.2)
\end{pspicture}$$

Let $\CA^{h}_n(\delta)$ be the rational vector space generated by oriented trivalent graphs with $2n$ vertices whose edges are oriented and equipped with some rational functions of $\QQ[t^{\pm 1},1/\delta(t)]$,
and quotiented by the following relations:
\begin{enumerate}
\item Reversing the orientation of an edge beaded by $P(t)$ and transforming this $P(t)$ into $P(t^{-1})$
gives the same element in the quotient.
$$\begin{pspicture}[shift=-0.6](-.5,-.2)(1,1)
\psline[linewidth=1.2pt]{-}(0,1.1)(.4,.9)
\psline[linewidth=1.2pt]{-}(.8,1.1)(.4,.9)
\psline[linewidth=1.2pt]{*-}(.4,.9)(.4,.35)
\psline[linewidth=1.2pt]{*->}(.4,0)(.4,.45)
\rput(-.05,.45){\small $P(t)$}
\psline[linewidth=1.2pt]{-}(0,-.2)(.4,0)
\psline[linewidth=1.2pt]{-}(.8,-.2)(.4,0)
\end{pspicture} = \begin{pspicture}[shift=-0.6](-.3,-.2)(1.6,1)
\psline[linewidth=1.2pt]{-}(-.2,1.1)(.2,.9)
\psline[linewidth=1.2pt]{-}(.6,1.1)(.2,.9)
\psline[linewidth=1.2pt]{-*}(.2,.55)(.2,0)
\psline[linewidth=1.2pt]{*->}(.2,.9)(.2,.45)
\rput(.8,.45){\small$P(t^{-1})$}
\psline[linewidth=1.2pt]{-}(-.2,-.2)(.2,0)
\psline[linewidth=1.2pt]{-}(.6,-.2)(.2,0)
\end{pspicture}$$
\item If two graphs only differ by the label of one oriented edge, that is $P(t)$
for one of them and $Q(t)$ for the other one,
then the class of their sum is the class of the same graph with label $(P(t)+Q(t))$.
$$
\begin{pspicture}[shift=-0.6](-1.7,-.2)(.8,1)
\psline[linewidth=1.2pt]{-}(0,1.1)(.4,.9)
\psline[linewidth=1.2pt]{-}(.8,1.1)(.4,.9)
\psline[linewidth=1.2pt]{*-}(.4,.9)(.4,.35)
\psline[linewidth=1.2pt]{*->}(.4,0)(.4,.45)
\rput(-.5,.45){\small $P(t)+Q(t)$}
\psline[linewidth=1.2pt]{-}(0,-.2)(.4,0)
\psline[linewidth=1.2pt]{-}(.8,-.2)(.4,0)
\end{pspicture} =
\begin{pspicture}[shift=-0.6](-.7,-.2)(.8,1)
\psline[linewidth=1.2pt]{-}(0,1.1)(.4,.9)
\psline[linewidth=1.2pt]{-}(.8,1.1)(.4,.9)
\psline[linewidth=1.2pt]{*-}(.4,.9)(.4,.35)
\psline[linewidth=1.2pt]{*->}(.4,0)(.4,.45)
\rput(-.05,.45){\small $P(t)$}
\psline[linewidth=1.2pt]{-}(0,-.2)(.4,0)
\psline[linewidth=1.2pt]{-}(.8,-.2)(.4,0)
\end{pspicture} +
\begin{pspicture}[shift=-0.6](-.7,-.2)(.8,1)
\psline[linewidth=1.2pt]{-}(0,1.1)(.4,.9)
\psline[linewidth=1.2pt]{-}(.8,1.1)(.4,.9)
\psline[linewidth=1.2pt]{*-}(.4,.9)(.4,.35)
\psline[linewidth=1.2pt]{*->}(.4,0)(.4,.45)
\rput(-.05,.45){\small $Q(t)$}
\psline[linewidth=1.2pt]{-}(0,-.2)(.4,0)
\psline[linewidth=1.2pt]{-}(.8,-.2)(.4,0)
\end{pspicture}$$
\item Multiplying by $t$ the three rational functions of three edges adjacent to a vertex, oriented towards that vertex, does not change the element in the quotient.
$$
\begin{pspicture}[shift=-1](-.5,-.5)(2,1.4)
\psline[linewidth=1.2pt]{*-}(.5,.5)(.5,-.4)
\psline[linewidth=1.2pt]{->}(.5,-.4)(.5,.05)
\rput(.15,.05){\small $P(t)$}
\psline[linewidth=1.2pt]{-}(-.4,1.4)(.5,.5)
\psline[linewidth=1.2pt]{->}(-.4,1.4)(.05,.95)
\rput(.35,1.15){\small $R(t)$}
\psline[linewidth=1.2pt]{-}(1.4,1.4)(.5,.5)
\psline[linewidth=1.2pt]{->}(1.4,1.4)(.95,.95)
\rput(1.3,.8){\small $Q(t)$}
\end{pspicture}
= \begin{pspicture}[shift=-1](-.5,-.5)(2,1.4)
\psline[linewidth=1.2pt]{*-}(.5,.5)(.5,-.4)
\psline[linewidth=1.2pt]{->}(.5,-.4)(.5,.05)
\rput(.1,.05){\small $tP(t)$}
\psline[linewidth=1.2pt]{-}(-.4,1.4)(.5,.5)
\psline[linewidth=1.2pt]{->}(-.4,1.4)(.05,.95)
\rput(.35,1.15){\small $tR(t)$}
\psline[linewidth=1.2pt]{-}(1.4,1.4)(.5,.5)
\psline[linewidth=1.2pt]{->}(1.4,1.4)(.95,.95)
\rput(1.4,.8){\small $tQ(t)$}
\end{pspicture}
\;\;\;\mbox{and}\;\;\; 
\begin{pspicture}[shift=-1](-.05,-.5)(1.05,1.4)
\psline[linewidth=1.2pt]{*-}(.5,.5)(.5,-.4)
\psline[linewidth=1.2pt]{->}(.5,-.4)(.5,.05)
\rput(.05,.05){\small $P(t)$}
\psarc[linewidth=1.2pt](.5,.9){.4}{85}{270}
\psarc[linewidth=1.2pt]{->}(.5,.9){.4}{-90}{90}
\rput(.5,.95){\small $Q(t)$}
\end{pspicture}=\begin{pspicture}[shift=-1](-.05,-.5)(1.05,1.4)
\psline[linewidth=1.2pt]{*-}(.5,.5)(.5,-.4)
\psline[linewidth=1.2pt]{->}(.5,-.4)(.5,.05)
\rput(.05,.05){\small $tP(t)$}
\psarc[linewidth=1.2pt](.5,.9){.4}{85}{270}
\psarc[linewidth=1.2pt]{->}(.5,.9){.4}{-90}{90}
\rput(.5,.95){\small $Q(t)$}
\end{pspicture}
$$
\item (AS) Changing the orientation of a vertex multiplies the element of the quotient by $(-1)$.
$$ \begin{pspicture}[shift=-0.4](0,-.2)(.8,.7)
\psset{xunit=.7cm,yunit=.7cm}
\psarc[linewidth=.5pt](.5,.5){.2}{-70}{15}
\psarc[linewidth=.5pt](.5,.5){.2}{70}{110}
\psarc[linewidth=.5pt]{->}(.5,.5){.2}{165}{250}
\psline[linewidth=1.5pt]{*-}(.5,.5)(.5,0)
\psline[linewidth=1.5pt]{-}(.1,.9)(.5,.5)
\psline[linewidth=1.5pt]{-}(.9,.9)(.5,.5)
\end{pspicture}
+
\begin{pspicture}[shift=-0.4](0,-.2)(.8,.7)
\psset{xunit=.7cm,yunit=.7cm}
\pscurve[linewidth=1.5pt]{-}(.9,.9)(.3,.7)(.5,.5)
\pscurve[linewidth=1.5pt,border=2pt]{-}(.1,.9)(.7,.7)(.5,.5)
\psline[linewidth=1.5pt]{*-}(.5,.5)(.5,0)
\end{pspicture}=0$$
\item (IHX or Jacobi) The sum of three graphs that coincide outside a disk, where they look as in the picture below, vanishes in the quotient.
(The complete edges of the relation are equipped with the polynomial $1$ that is not written.)
$$
\begin{pspicture}[shift=-0.4](0,-.2)(.8,1)
\psset{xunit=.7cm,yunit=.7cm}
\psline[linewidth=1pt]{-*}(.1,1)(.35,.2)
\psline[linewidth=1pt]{*-}(.5,.5)(.5,1)
\psline[linewidth=1pt]{-}(.75,0)(.5,.5)
\psline[linewidth=1pt]{-}(.25,0)(.5,.5)
\end{pspicture}
+
\begin{pspicture}[shift=-0.4](0,-.2)(.8,1)
\psset{xunit=.7cm,yunit=.7cm}
\psline[linewidth=1pt]{*-}(.5,.6)(.5,1)
\psline[linewidth=1pt]{-}(.8,0)(.5,.6)
\psline[linewidth=1pt]{-}(.2,0)(.5,.6)
\pscurve[linewidth=1pt,border=2pt]{-*}(.1,1)(.3,.3)(.7,.2)
\end{pspicture}
+
\begin{pspicture}[shift=-0.4](0,-.2)(.8,1)
\psset{xunit=.7cm,yunit=.7cm}
\psline[linewidth=1pt]{*-}(.5,.35)(.5,1)
\psline[linewidth=1pt]{-}(.75,0)(.5,.35)
\psline[linewidth=1pt]{-}(.25,0)(.5,.35)
\pscurve[linewidth=1pt,border=2pt]{-*}(.1,1)(.2,.75)(.7,.75)(.5,.85)
\end{pspicture}
=0
\;\;\mbox{where} \begin{pspicture}[shift=-0.4](-.1,-.2)(.8,.7)
\psset{xunit=.7cm,yunit=.7cm}
\psline[linewidth=1pt]{-}(0,1.1)(.4,.9)
\psline[linewidth=1pt]{-}(.8,1.1)(.4,.9)
\psline[linewidth=1pt]{*-*}(.4,.9)(.4,0)
\rput(.2,.45){\small $1$}
\psline[linewidth=1pt]{-}(0,-.2)(.4,0)
\psline[linewidth=1pt]{-}(.8,-.2)(.4,0)
\end{pspicture} =\begin{pspicture}[shift=-0.4](-.1,-.2)(.8,.7)
\psset{xunit=.7cm,yunit=.7cm}
\psline[linewidth=1pt]{-}(0,1.1)(.4,.9)
\psline[linewidth=1pt]{-}(.8,1.1)(.4,.9)
\psline[linewidth=1pt]{*-*}(.4,.9)(.4,0)
\psline[linewidth=1pt]{-}(0,-.2)(.4,0)
\psline[linewidth=1pt]{-}(.8,-.2)(.4,0)
\end{pspicture}.
$$
\end{enumerate}

For example, there is a map $\psi \colon \CA^{h}_1(\delta)
\rightarrow \QQ(\fvar,\svar)$ such that
 $$\psi(\begin{pspicture}[shift=-.5](-.3,0)(2.3,1.6)
\psline[linewidth=1.2pt]{-*}(0,.6)(2,.6)
\psline[linewidth=1.2pt]{*->}(0,.6)(1,.6)
\rput[b](1,.7){\small $Q(t)$}
\psecurve[linewidth=1.2pt]{-*}(1,1.2)(0,.6)(1,0)(2,.6)(1,1.2)
\psecurve[linewidth=1.2pt]{*->}(1,1.2)(0,.6)(1,0)(2,.6)
\rput[b](1,.1){\small $R(t)$}
\psecurve[linewidth=1.2pt]{-*}(1,0)(0,.6)(1,1.2)(2,.6)(1,0)
\psecurve[linewidth=1.2pt]{*->}(1,0)(0,.6)(1,1.2)(2,.6)
\rput[b](1,1.25){\small $P(t)$}
\end{pspicture})=\sum_{\mathfrak{S}_3(x,y,z)}(P(\fvar)Q(\svar)R(\tvar)+P(\fvar^{-1})Q(\svar^{-1})R(\tvar^{-1}))$$ for elements $P(t)$, $Q(t)$ and $R(t)$ of $\QQ[t^{\pm 1},1/\delta(t)]$, where $\tvar=(\fvar\svar)^{-1}$.

A more general version of this space of diagrams was introduced in
\cite[Definition 3.8]{GK}.

\begin{remark}
As pointed out to me by the referee, if we replaced $\QQ[t^{\pm 1},1/\delta(t)]$ by $\QQ[t^{\pm 1},1/\delta(t),1/(t-1)]$ in the definition of $\CA_n^h(\delta)$, unwanted relations would occur. For example, diagrams with separating edges would become trivial. Indeed, such a diagram with a separating edge beaded by $P/(1-t)$ would be equal to the diagram obtained by changing the beading into $Pt/(1-t)$ so that the difference of these two diagrams would be trivial. Therefore, the dumbbell-shaped diagram \begin{pspicture}[shift=-.3](-.2,0)(1.7,.8) 
\psarc(0.2,0.4){.3}{-190}{0}
\psarc{->}(0.2,0.4){.3}{0}{180}
\psline{*-*}(0.5,.4)(1,.4) 
\pscircle(1.3,0.4){.3}
\psarc(1.3,0.4){.3}{-10}{180}
\psarc{->}(1.3,0.4){.3}{-180}{0}
\rput(.1,.4){\small $t$}
\rput(1.4,.4){\small $t$}
\end{pspicture}, that does not vanish in $\CA_1^h(\delta)$, would vanish.
\end{remark}

\subsection{On compactifications of configuration spaces}
\label{subcompconf}

We describe the main features of the Fulton and MacPherson type compactifications of configuration spaces \cite{fmcp} that are used
to define configuration space invariants of knots or manifolds.
Details can be found in \cite[Section 3]{lesconst}.

Let $N$ be a finite set, and let $I$ be a subset of $N$ with cardinality $\sharp I \geq 2$. Let $\Delta_I(M^N)=\{(m_i)_{i \in N}; m_i=m_j\;\; \mbox{if}\;\; \{i,j\} \subset I\}$.
$\Delta_I(M^N)$ is a codimension $3(\sharp I -1)$-manifold.
The fiber of its unit normal bundle at $(m_j)_{j \in N}$ is the space 
$(T_{m_i}M^I/\Delta_I(T_{m_i}M^I) \setminus \{0\})/\RR^{+\ast}$ where $i \in I$. It is the space $C_I(T_{m_i}M)$
of $I$-tuples of points of the tangent space $T_{m_i}M$ up to global translation and global homothety with positive ratio.
Thus, in the blown-up manifold $M^N(\Delta_I(M^N))$ (see Definition~\ref{defblowup}), a point $P$ that projects to $\Delta_I(M^N)$ under the canonical projection is equipped with the data of the infinitely small configuration $c_I(P) \in C_I(T_{m_i}M)$ of the points $m_i$ indexed by $I$ in $T_{m_i}M$ (up to translation and dilation).

Of course, the manifold $C_2(M)$ reads $M^2(\Delta_{\{1,2\}}(M^2))$ with this notation.

Let $J$ be a subset of $I$, $J \subsetneq I$, $\sharp J \geq 2$. $$\Delta_I(M^N) \subset \Delta_J(M^N)$$
The closure of $\Delta_J(M^N) \setminus \Delta_I(M^N)$ in $M^N(\Delta_I(M^N))$ is the submanifold $$\Delta_J(M^N)(\Delta_I(M^N)),$$ and we can blow up $M^N(\Delta_I(M^N))$ along $\Delta_J(M^N)(\Delta_I(M^N))$ to get a manifold with boundary and with corners, where the relative configuration of the possibly coinciding points of $T_{m_i}M$ indexed by elements of $J$ in an infinitely small configuration $c_I(P)$ of points of $T_{m_i}M$ is known. It is thought of as infinitely smaller.

\begin{definition} 
\label{defcgamma} Let $\Gamma$ be a graph whose vertices are indexed by $\{1,2,\dots,2n\}$. 
When $I$ is a subset of $\{1,2,\dots,2n\}$, the graph $\Gamma_I$ is the subgraph of $\Gamma$ made of the vertices indexed in $I$ and the edges between two of them.
The {\em configuration space $C(M,\Gamma)$\/} is 
the smooth compact $6n$-manifold with corners obtained by blowing up successively all the $\Delta_I(M^N)$ for which $\Gamma_I$ is connected and $\sharp I \geq 2$, inductively, starting with maximal not yet treated $I$ with this property. The space $C(M,\Gamma)$ is a compactification of the complement of all the diagonals of $M^{2n}$,
with a canonical projection $p(M,\Gamma)\colon C(M,\Gamma) \rightarrow M^{2n}$.
The {\em configuration space $C_{2n}(M)$\/} is $C(M,\Gamma_c)$
where $\Gamma_c$ is the complete graph with $2n$ vertices.
\end{definition}

Let $P_0$ be a point of $C(M,\Gamma)$, it projects to $(m_i)_{i \in \{1,2,\dots,2n\}}$. Consider the maximal connected subgraphs $\Gamma_I$ of $\Gamma$ (with $\sharp I \geq 2$) such that $m_j=m_k$ whenever $j$ and $k$ are in $I$. For each of these subgraphs $\Gamma_I$, consider the maximal connected subgraphs $\Gamma_J$ with $\sharp J \geq 2$ and $J \subsetneq I$ such that the points indexed by $J$ coincide in the infinitely small configuration $c_I(P_0)$.
Iterating the process, the ``considered'' graphs form a family
$\CE(P_0)$ of connected subgraphs $\Gamma_I$ of $\Gamma$ such that the intersection $\Gamma_I \cap \Gamma_J$ of any two graphs of the family is $\Gamma_I$, $\Gamma_J$ or $\emptyset$.
The points $P$ such that $\CE(P)$ is empty are the points in the interior of $C(M,\Gamma)$. If $\CE(P_0)\neq \emptyset$, the points  $P$ such that $\CE(P)=\CE(P_0)$ are in the boundary of $C(M,\Gamma)$ and they form a codimension $\sharp \CE(P_0)$ {\em face\/} of $C(M,\Gamma)$.
In particular, the {\em codimension $1$ faces\/} of $C(M,\Gamma)$ correspond to connected subgraphs $\Gamma_I$ of $\Gamma$.

When $\Gamma$ has no {\em loops\/} (edges whose two ends coincide), for every oriented edge $e$ of $\Gamma$ from the vertex $v(j)$ of $\Gamma$ indexed by $j$ to $v(k)$, there is a canonical projection
$$p(\Gamma,e) \colon C(M,\Gamma) \rightarrow C_2(M)$$
that maps a point $P$ of $p(M,\Gamma)^{-1}((m_i)_{i \in \{1,2,\dots,2n\}})$ to $(m_j,m_k)$ if $m_j \neq m_k$, and to the infinitely small configuration of $(T_{m_j}M^{\{j,k\}}/\Delta_{\{j,k\}}(T_{m_j}M^{\{j,k\}}) \setminus \{0\})/\RR^{+\ast}$ that can be seen at some scale in $C(M,\Gamma)$, if $m_j = m_k$.

\subsection{Definition of the invariants $\tilde{z}_n(M,\KK)$}

Fix $(M,\KK,\tau)$ as in Subsection~\ref{subdefbord}
so that $\tau$ maps the tangent vectors of $K$ to $\RR \qvarM$, where $\qvarM$ is the vertical vector of $S^2$ pointing upward. For any $n \in \NN \setminus \{0\}$, the invariant $\tilde{z}_n(M,\KK,\tau)$ is defined as follows.
Consider 
\begin{itemize}
\item $\varepsilon \in ]0,\frac{1}{2n}[$, $\funcM$ (involved in the definition of the map $\pi$ in Subsection~\ref{subdefbord}),
\item $3n$ disjoint parallels $(K_1, \dots, K_{3n})$ of $K$ on $ \partial M_{[0,1]}$, with respect to the parallelisation of $\KK$,
\item a {\em regular\/} $3n$-tuple $(\fvarM_1, \dots,\fvarM_{3n})$  of $(S^2_H)^{3n}$, where {\em regular\/} means in some open dense subset of $(S^2_H)^{3n}$ that will be specified in Subsection~\ref{subuplegen},
\item $3n$ integral chains (integral combinations of properly $C^{\infty}$--embedded $4$-simplices) $H_i$ in {\em general $3n$-position\/} in $\tilde{C}_2(M_{[0,2]})$  whose boundaries are $$k\delta(t)
\left(\pi_{|\partial \tilde{C}_2(M_{[0,2]}) \setminus \partial \tilde{C}_2(M_{[0,2[})}^{-1}(X_i) \cup s_{\tau}(M_{[0,2]};X_i) \cup (-J_{\Delta}) ST(M)_{|K_i}\right),$$ respectively, for some positive integer $k$.
The notion of {\em general $3n$-position\/} will be specified in Definition~\ref{defcyclegen}.
\end{itemize}

The chains $H_i$ will be seen as cycles of $(C_2(M_{[0,2]}),\partial C_2(M_{[0,2]}))$ (combinations of simplices) with coefficients in $\QQ[H_1(M)/\mbox{Torsion}]$ by projection as follows.
Pick a basepoint $\ast_C$ in $ST(M_{[0,2]}) \subset \partial {C}_2(M)$, and consider the equivalence relation on the set of paths of $C_2(M)$ from $\ast_C$ 
to some other point in $C_2(M)$ that identifies two paths if they lift to paths with identical ends in $\TCM$. For a point of $C_2(M)$ that projects to $(m_1,m_2) \in M^2$, this datum is equivalent to a rational homology class of paths from $m_1$ to $m_2$.
A {\em chain of $C_2(M)$ with coefficients\/} in $\QQ[H_1(M)/\mbox{Torsion}]$ is a chain every point $P$ of which is equipped with a class of paths from $\ast_C$ to $P$ as above, in a continuous way (or every simplex (or contractible cell) of which is equipped with such a class of paths in a way compatible with face identifications).

Let $\CS_n$ be the set of connected trivalent graphs $\Gamma$ with $2n$ vertices numbered from $1$ to $2n$ and with $3n$ oriented edges numbered from $1$ to $3n$ without loops.

Consider $\Gamma \in \CS_n$.
Set $C(M_{[0,4]},\Gamma)=p(M,\Gamma)^{-1}(M_{[0,4]}^{2n}) \subset C(M,\Gamma)$. See Definition~\ref{defcgamma}.
Let $e=e(i)$ be the edge of $\Gamma$ numbered by $i$ that goes from  $v(j)$ to $v(k)$. Consider the map $p(\Gamma,i)=p(\Gamma,e(i))\colon C(M,\Gamma) \rightarrow C_2(M)$ that lifts 
$(m_1, \dots, m_{3n}) \mapsto (m_j,m_k)$, continuously.
Let $H_i(\Gamma) \subset C(M_{[0,4]},\Gamma)$ be defined as
$$H_i(\Gamma)=p(\Gamma,i)^{-1}\left(H_i \cup k\delta(t) \pi^{-1}(X_i)\right)$$ 
where $\pi^{-1}(X_i)=\pi_{|\tilde{C}_2(M_{[0,4]}) \setminus  \tilde{C}_2(M_{[0,2[})}^{-1}(X_i) \subset \tilde{C}_2(M_{[0,4]}) \setminus \tilde{C}_2(M_{[0,2[})$ and $H_i$ are seen as chains with coefficients. Then $H_i(\Gamma)$ is a chain (here, a combination of manifolds with boundaries) whose points (that project to) $(m_1, \dots, m_{3n})$ are continuously equipped with rational homology classes of paths from $m_j$ to $m_k$. It is cooriented in $C(M,\Gamma)$ (that is oriented like $M^{2n}$) by the coorientation of $H_i$ in $C_2(M)$.

Then the intersection of the $H_i(\Gamma)$, for $i \in \{1,2,\dots, 3n\}$, is a compact subspace $I(\Gamma,\{H_j\})$ of $C(M_{[0,4]},\Gamma)$. Its image $p(\Gamma,i)(I(\Gamma,\{H_j\}))$ in $H_i\cup \pi^{-1}(X_i)$ is a compact subset of $H_i\cup \pi^{-1}(X_i)$.

\begin{definition}
 \label{defcyclegen}
The $H_i$ are said to be in {\em general $3n$-position\/} if, for any $\Gamma \in \CS_n$, 
\begin{itemize}
\item $I(\Gamma,\{H_j\})$ is finite,
\item $p(\Gamma,i)(I(\Gamma,\{H_j\}))$ is made of points in the interiors of the $4$-simplices of $H_i$ or in the interior of $\pi^{-1}(X_i)$, for any $i \in \{1,2,\dots, 3n\}$, and, 
\item all the $H_i(\Gamma)$ intersect transversally at the corresponding points that are in the interior of $C(M_{[0,4]},\Gamma)$.\end{itemize}
\end{definition}

The fact that such $H_i$ exist will be proved
in Subsection~\ref{subtrans}.

Fix $\Gamma \in \CS_n$.
Under the given assumptions,
the cooriented $H_i(\Gamma)$ only intersect transversally at distinct points 
in the interior of $C(M_{[0,4]},\Gamma)$.
We define their equivariant algebraic intersection $I_{\Gamma}(\{H_i\}) \in \CA^{h}_n(\delta)$ as follows.

Consider an intersection point $m$, such that $p(M,\Gamma)(m)=(m_1, \dots, m_{2n}) \in M^{2n}$, it is equipped with a sign $\varepsilon(m)$ as usual, and it is also equipped with the following additional data:
Associate $m_j$ with $v(j)$. Since $m$ belongs to $H_i(\Gamma)$, the edge $e(i)$ from $v(j)$ to $v(k)$ is equipped with a rational homology class of paths from $m_j$ to $m_k$. Each edge of $\Gamma$ is equipped with a rational homology class of paths between its ends in this way.
Choose a simply connected graph $\gamma$ (e. g. a spider) in $M$ that contains a basepoint and all the points $m_j$ of $M$. A path from $m_j$ to $m_k$ can be composed with paths in the graph $\gamma$ to become a loop based at the basepoint whose rational homology class is well-determined by the rational homology class of the path and by $\gamma$. It reads $n(m,e(i),\gamma)[K]$
for some integer $n(m,e(i),\gamma)$.
Let $\Gamma(m,\gamma)$ be the graph obtained from $\Gamma$ by assigning $t^{-n(m,e(i),\gamma)}/(k\delta(t))$ to the oriented edge $e(i)$ for each edge. Orient the vertices of $\Gamma$ so that the permutation of the half edges from (first half of first edge, second half of first edge, \dots, second half of last edge) to (half-edges of the first vertex
ordered in a way compatible with the vertex orientation, \dots) is even.

Note that the class of $\Gamma(m,\gamma)$ in $\CA^{h}_n(\delta)$ does not depend on $\gamma$. Indeed, changing the path that goes from $m_j$ to the basepoint in $\gamma$
amounts to add some fix integer to $n(m,e,\gamma)$ for the edges $e$ going to $m_j$ and to remove this integer from $n(m,e,\gamma)$ for the edges starting at $m_j$. Because of Relations 1 and 3, it does not change the class of $\Gamma(m,\gamma)$ that will be denoted by $[\Gamma(m)]$.
Finally, assign $\varepsilon(m)[\Gamma(m)] \in \CA^{h}_n(\delta)$ to $m$.

Then define $I_{\Gamma}(\{H_i\}) \in \CA^{h}_n(\delta)$ as the sum over the intersection points of the $\varepsilon(m)[\Gamma(m)]$.

\begin{theorem}
\label{thmhighloop}
$$\tilde{z}_n(M,\KK,\tau)=\sum_{\Gamma \in \CS_n}\frac{I_{\Gamma}(\{H_i\})}{2^{3n}(3n)!(2n)!} \in \CA^{h}_n(\delta)$$ is an invariant of $(M,\KK,\tau)$.
Let $\xi_n$ be the element of $\CA_n^h(1)$ defined in \cite[Section 1.6]{lesconst}. It is an element of the subspace of $\CA_n^h(\delta)$ generated by the diagrams whose edges are beaded by $1$, and it is zero when $n$ is even.
Then $$\tilde{z}_n(M,\KK)=\tilde{z}_n(M,\KK,\tau)+\frac{p_1(\tau)}{4}\xi_n$$ is an invariant of $(M,\KK)$.
\end{theorem}

The collection $(\Delta(M),(\tilde{z}_n(M,\KK))_{n\in \NN})$ should be equivalent to the Kricker rational lift of the Kontsevich integral.

Note that with the map $\psi$ of the end of Subsection~\ref{subdiagbead} we have
$$ \CQ(M,\KK)= \psi(\tilde{z}_1(M,\KK)).$$
The properties of the chains $H_i$ that are proved in \cite{betaone} will allow us to generalize properties of $\CQ(M,\KK)$ to $(\tilde{z}_n(M,\KK))_{n\in \NN}$ in a future work.

\begin{remark} Though the graphs $\Gamma$ of $\CS_n$ in Theorem~\ref{thmhighloop} have no loops, it is shown in \cite[Section 9.2]{betaone} that the evaluation of $\tilde{z}_1$ on a dumbbell-shaped clasper, whose looped edges are equipped with the equivariant linking numbers of the corresponding leaves, is the corresponding element \begin{pspicture}[shift=-.3](-.2,0)(1.7,.8) 
\psarc(0.2,0.4){.3}{-190}{0}
\psarc{->}(0.2,0.4){.3}{0}{180}
\psline{*-*}(0.5,.4)(1,.4) 
\pscircle(1.3,0.4){.3}
\psarc(1.3,0.4){.3}{-10}{180}
\psarc{->}(1.3,0.4){.3}{-180}{0}
\rput(.1,.4){\small $P$}
\rput(1.4,.4){\small $Q$}
 \end{pspicture} of $\CA^1_h(\delta)$. Similar properties are expected for all the $\tilde{z}_n$.
\end{remark}

\subsection{Regular points of $(S^2_H)^{3n}$}
\label{subuplegen}
The restriction of $p$ to $$\left(\tilde{C}_2(M_{[0,4]}) \setminus \tilde{C}_2(M_{[0,2[})\right) \cap \pi^{-1}(\overline{S^2_H})$$ is a diffeomorphism $p_H$ onto its image $P_H$ in $\left({C}_2(M_{[0,4]}) \setminus {C}_2(M_{[0,2[})\right)$.
Define $\overline{\pi}$ as $\pi \circ p_H^{-1}$ on $P_H$.

Consider a graph $\Gamma \in \CS_n$, its set of vertices $V(\Gamma)$, a coloring $L\colon V(\Gamma) \rightarrow \{M_{[0,2]},M_{[2,4]}\}$ of its vertices, and a set $B$ of edges of $\Gamma$ that contains all the edges between two vertices of color
$M_{[0,2]}$. Set $A=\{1,\dots,3n\} \setminus B$ and 
$${C}(\Gamma,L,B)=p(M,\Gamma)^{-1}\left(M_{[0,2]}^{L^{-1}(\{M_{[0,2]}\})} \times M_{[2,4]}^{L^{-1}(\{M_{[2,4]}\})}\right) \cap \cap_{a\in A}p(\Gamma, a)^{-1}(P_H).$$
Consider the product $$g(\Gamma,L,B)= \prod_{a\in A}\overline{\pi} \circ p(\Gamma, a)\times \mbox{Identity}(S^2)^B \colon {C}(\Gamma,L,B)\times (S^2)^B \rightarrow (S^2)^A \times (S^2)^B.$$

A {\em regular value\/} of this map $g(\Gamma,L,B)$ is a point $Y$ of the target
$(S^2)^{3n}$ such
that for any point $y$ of $g(\Gamma,L,B)^{-1}(Y)$, for any face $\CF$ of ${C}(\Gamma,L,B)$, the tangent map of $g(\Gamma,L,B)$ at $y$ is surjective, and, if $y\in \CF \times (S^2)^B$, the tangent map of $g(\Gamma,L,B)_{|\CF \times (S^2)^B}$ at $y$ is surjective. 

\begin{lemma}
 The set of regular values of $g(\Gamma,L,B)_{|g(\Gamma,L,B)^{-1}((S_H^2)^{3n})}$ is open and dense in $(S_H^2)^{3n}$.
\end{lemma}
\begin{proof}
The density is a direct corollary of the Morse-Sard theorem \cite[Chapter 3, Section 1]{hirsch}. Since
the set where the tangent maps are not all surjective is closed in the compact source of $g(\Gamma,L,B)$, its image is compact, and the set of regular values is open.
\end{proof}

Thus the finite intersection of the sets of regular values of such maps $g(\Gamma,L,B)$ is also dense and open.

\begin{definition}
 \label{defuplegen}
Let $P_V(X_i)$ be the vertical plane of $\RR^3$ that contains $X_i$.
An element $(X_i)_i$ of $(S^2_H)^{3n}$ is {\em regular\/} if it is a regular value for all the maps $g(\Gamma,L,B)$, and if, for any pair $\{i,j\}$ of $\{1,2,\dots,3n\}$, $P_V(X_i) \cap P_V(X_j)$ is the vertical line. 

\end{definition}
The set of regular $(X_i)_i$ is dense and open in $(S^2_H)^{3n}$.
\begin{lemma}
\label{lemintboun}
If $(\fvarM_1, \dots,\fvarM_{3n})$ is regular, then 
 the $H_i(\Gamma)$ do not intersect over the small diagonal of $M^{2n}$ in $C_{2n}(M)$.
\end{lemma}
\begin{proof} If they did, there would be an intersection point in the small diagonal of $C_{2n}(M)$.
Such a point would be a configuration in the tangent space $T_m M$ of $M$ at some point $m$ of $M$, where $T_m M$ is identified with $\RR^3$ via $\tau$. Let $D$ be a $3$--ball in $M_{[3, 3+\varepsilon]}$.
Note that $C(D,\Gamma)$ embeds in $C(\RR^3,\Gamma)$ and that $\overline{\pi} \circ p(\Gamma, a)$ reads as the map ``direction of $e(a)$'', there. This map factors through 
the global translations and the homotheties with positive ratio.
Let $L_{[2,4]}$ be the constant coloring with value $M_{[2,4]}$.
Since $(\fvarM_1, \dots,\fvarM_{3n})$ is regular, it belongs neither
to the --at most $(6n-4)$-dimensional-- $p(\Gamma , L_{[2,4]},\emptyset)(C(D,\Gamma) \cap {C}(\Gamma,L_{[2,4]},\emptyset))$, nor to the --at most $(6n-2)$-dimensional-- $p(\Gamma , L_{[2,4]},\{e(i)\})((C(D,\Gamma)\cap {C}(\Gamma,L_{[2,4]},\{e(i)\}))\times S^2)$ for some edge e(i), when $m$ belongs to $K_i$. 
\end{proof}

\subsection{Transversality}
\label{subtrans}

In this subsection, we prove that there exist $H_i$ in general $3n$-position.

It follows from general transversality properties that the $H_i(\Gamma)$ can be perturbed so that they have a finite number of transverse intersection points (see \cite[Chapter 3]{hirsch}). The subtleties here are
\begin{itemize}
\item  that we want simultaneous transversality for all the $\Gamma \in \CS_n$,
\item that we would like to perturb the $H_i$ rather than the $H_i(\Gamma)$\footnote{The $H_i(\Gamma)$ are associated with $H_i$ and $\Gamma$, and the face identifications, that will be performed later,  could not be performed without the fact that $H_i(\Gamma)$ is precisely the preimage of some fixed $H_i \cup k\delta(t) \pi^{-1}(X_i)$.}.
\end{itemize}

To prove the existence of $H_i$ in general $3n$-position, we proceed as follows.

Choose cycles $H^0_i$, together with their simplicial decompositions, with the given boundaries, so that the $H^0_i$ are transverse to $\partial \tilde{C}_2(M_{[0,2]})$.

Assume that
\begin{itemize}
 \item the $4$-dimensional simplices of $H^0_i$ are $C^{\infty}$ embeddings indexed by elements $j$ of some finite set $S(i)$, and equipped with local coefficients,
 $$\Delta^0_j(i) \colon \Delta^4 \rightarrow O_j(i) \subset C_2(M_{[0,2]}) $$
of the standard $4$-simplex $\Delta^4$ in open regions $O_j(i)$ of $C_2(M_{[0,2]})$ such that there are $C^{\infty}$ embeddings 
$$\Phi_j(i) \colon O_j(i) \rightarrow \RR^6$$ with open image in some intersection of half-spaces in $\RR^6$,
\item there is a set $B_i$ of $3$-dimensional faces of the $4$-dimensional simplices of $H^0_i$ that constitute a simplicial decomposition of $\partial H^0_i$,
\item there is a bijection $b_i$ of the set of the $3$-dimensional faces of the $\Delta^0_j(i)$ that are not in $B_i$ to $A_i\times\{+,-\}$, for some finite set $A_i$, such that\\ for any $a \in A_i$, the face labeled by $(a,+)$ has the same image in $C_2(M)$ than the face labeled by $(a,-)$ with opposite boundary orientation and matching local coefficients (this is the {\em face identification\/}),
\item the simplices $\delta^0_r(i)$ of dimension $k$ ($k \leq 4$) of $H^0_i$ are codimension $0$ submanifolds with boundaries of the preimage of a regular point $N_{i,r}$ under a constraint map $f(\delta_r(i))$ from the intersection of the open sets $O_j(i)$ such that $\delta^0_r(i)$ is a face of $\Delta^0_j(i)$ to $S^{6-k}$.
\end{itemize}

Define $\CM(H^0_i)$ as the set of deformations 
$H_i=\{\Delta_j(i) \colon \Delta^4 \rightarrow O_j(i)\}_{j\in S(i)}$ of the $H^0_i=\{\Delta^0_j(i) \}_{j\in S(i)}$, where the set $S(i)$ of simplices, the face identification and the simplices of $\partial H_i$ are fixed.

Equip $\CM(H^0_i)$ with the $C^{\infty}$ topology that is its topology
of a subset of $$\prod_{j \in S(i)}C^{\infty}(\Delta^4,O_j(i))$$ where the topology of $C^{\infty}(\Delta^4,O_j(i))$ is defined in \cite[Chapter 2, Section 1]{hirsch}.

Consider the topological space
$$\CM=\prod_{i=1}^{3n}\CM(H^0_i).$$

Define the subset $\CM_{\ast}$ of $\CM$ made of the $(H_1, H_2, \dots, H_{3n})$ whose simplices satisfy:\\
For each $\Gamma \in \CS_n$, for any $3n$-tuple $(\delta_{j_i}(i) \subset H_i \cup \pi^{-1}(X_i))_{i \in \{1,\dots,3n\}}$ of simplices, where $\pi^{-1}(X_i)$ and $\partial \pi^{-1}(X_i)$ are abusively considered as additional simplices $\Delta_{0}(i)$ and $\partial \Delta_{0}(i)$ and the constraint map associated to $(\Delta_{0}(i) \subset (O_0(i)=\overline{\pi}^{-1}(S^2_H)))$ is $\overline{\pi}$,\\
$\ast(\Gamma,\delta_{j_i}(i))$: the $p(\Gamma,i)^{-1}(\delta_{j_i}(i))$ are transverse in $C(M_{[0,4]},\Gamma)$.

\begin{proposition}
 $\CM_{\ast}$ is open and dense in $\CM$. In particular, there exists a $(3n)$-tuple $(H_1, H_2, \dots, H_{3n})$ in general $3n$-position.
\end{proposition}
\begin{proof} It is enough to prove that the subspace of $\CM$ made of the $(H_1, \dots, H_{3n})$ that satisfy $\ast(\Gamma,\delta_{j_i}(i))$ is open and dense for each $\Gamma \in \CS_n$, for any $3n$-tuple $(\delta_{j_i}(i) \subset H_i)_{i \in \{1,\dots,3n\}}$ of (generalized) simplices.

It is obviously open and we only need to prove that it is dense.

Consider the regions where the $\delta_{j_i}(i)_{i \in \{1,\dots,3n\}}$
live. The intersection of their preimages under the $p(\Gamma,i)$ determines an open subspace of $$A_I= \cap_{i\in I}p(\Gamma,i)^{-1}(C_2(M_{[0,4]})\setminus C_2(M_{[0,2[})) \cap \cap_{j\in J}p(\Gamma,j)^{-1}(C_2(M_{[0,2]}))$$ for two complementary subsets $I$ and $J$ of $\{1,2,\dots,3n\}$,
where the transversality of the $\delta_{j_i}(i)$ reads:
$(N_{i,j_i})_{i \in \{1,\dots,3n\}}$ is a regular value of the constraint map $\prod_{i=1}^{3n}{f(\delta_{j_i}(i))\circ p(\Gamma,i)}$.
According to the Morse-Sard theorem, the set of regular $(N_i)_{i \in \{1,\dots,3n\}}$ is dense. It is furthermore open.

Therefore, if none of the $\delta_{j_i}(i)$ touches the boundary, it is easy to slightly move the simplices by changing the preimage of $N_{i,j_i}$ to the preimage of some closed point. (This can be done by a global isotopy of $C_{2}(M_{[0,2]})$ supported near $\delta_{j_i}(i)$
that moves $\delta_{j_i}(i)$ to the preimage of some close point, when $j_i \neq 0$.)

In general, we have a set of simplices in the boundary, with indices in $I_b$, a set of simplices that don't douch the boundary with indices in $I_e$, and the other ones with indices in $I_t$ that have a maximal face $\CF_i$ in the boundary. (Here, the special simplex $\pi^{-1}(X_i)$ is considered in the boundary.)
Since $(\fvarM_1, \dots,\fvarM_{3n})$ is regular, the intersection of the $p(\Gamma,i)^{-1}(\partial H_i \cup k\delta(t) \pi^{-1}(X_i))$, for $i \in I_b \cup I_t$, is transverse in the faces of
$A_I$.
Therefore, it is a $(6n-2(\sharp I_b +\sharp I_t)-c)$--submanifold of the codimension $c$ faces of $A_I$.
The intersection of the $\delta_{j_i}(i)$, for $i \in I_e$ will be transverse to this manifold if the restriction of
$\prod_{i\in I_e}{f(\delta_{j_i}(i))\circ p(\Gamma,i)}$ to this manifold is regular, and this can be achieved as before.
Then the transversality holds in a neighborhood of the $\CF_i$, and we can make it hold for all the simplices by applying the above argument to $(I_b,I_e \cup I_t)$ instead of $(I_b \cup I_t,I_e)$, the openness allowing us to recorrect the $\delta_{j_i}(i)$ for $i \in I_t$ near the boundaries without losing transversality.
\end{proof}

\subsection{Proof of invariance}

In this section, we prove that we have defined
a topological invariant of $(M,\KK,\tau)$.
We first prove that our construction yields a well-defined function $\tilde{z}_n$ of $((M,M_{[1,8]}),\KK,\tau,\funcM,\varepsilon,(K_i)_i)$, by showing that our definition is independent of $(X_1, \dots, X_{3n})$ and $(H_1,\dots,H_{3n})$, i.e. that it is independent of $(H_1,\dots,H_{3n})$ when $(X_1, \dots, X_{3n})$
is not fixed.

Since the invariance is obvious when $(H_1,\dots,H_{3n})$ moves in a neighborhood of $(H_1,\dots,H_{3n})$ in general $3n$-position, it is enough to prove invariance when $H_1$ is changed to some $H^{\prime}_1$
such that $(H^{\prime}_1, H_2, \dots, H_{3n})$ is in general $3n$-position.

First pick a $C^{\infty}$-path $P(X_1,X^{\prime}_1)$ from $X_1$ to $X^{\prime}_1$ in $S^2_H$ whose image is some codimension $0$ submanifold with boundary of a submanifold defined by $f_P(x)=0$ for some smooth map $f_P$ that is defined in a regular neighborhood of the image of this path and valued in $[-1,1]$.
Without loss assume that $(0,X_2,X_3,\dots, X_{3n})$ is a regular value for all the compositions of the maps $g(\Gamma,L,B)$ of Section~\ref{subuplegen} by $f_P$ on the first coordinate.
(If it is not the case, use the Morse-Sard theorem to find a close regular value and move everything slightly.)

\begin{lemma}
There exists a $5$-chain $C_0$ of $\tilde{C}_2(M_{[0,2]})$ in general $3n$-position with $H_2$,
\dots, $H_{3n}$ such that $$\partial C_0 = H^{\prime}_1-H_1 + k\delta(t)\left(s_{\tau}(M_{[0,2]};P(X_1,X^{\prime}_1)) - 
\pi_{|\partial \tilde{C}_2(M_{[0,2]}) \setminus \partial \tilde{C}_2(M_{[0,2[})}^{-1}(P(X_1,X^{\prime}_1))\right).$$
\end{lemma}
\begin{proof} The wanted boundary $\partial C_0$ is a $4$-cycle and $H_4(C_2(M_{[0,2]});\QQ[t^{\pm 1}])=0$.
\end{proof}

Here the notion of general $3n$-position is similar to the previous notion and can be achieved in the same way. Let $C=C_0 \cup k\delta(t) \pi^{-1}(P(X_1,X^{\prime}_1))$.

The intersection of $p(\Gamma,1)^{-1}(C)$ and the $H_i(\Gamma)$ for $i \geq 2$ is some $C^{\infty}$ compact $1$-manifold with boundary in $C(M,\Gamma)$.
$p(\Gamma,1)^{-1}(C)$ is locally seen as the zero locus of some real function except near $p(\Gamma,1)^{-1}(\partial C)$
where
some additional independent real function must be positive.
Therefore, 
$$\partial (p(\Gamma,1)^{-1}(C) \cap \bigcap_{i=2}^{3n}H_i(\Gamma))=$$
$$H^{\prime}_1(\Gamma) \cap \bigcap_{i=2}^{3n}H_i(\Gamma) - H_1(\Gamma) \cap \bigcap_{i=2}^{3n}H_i(\Gamma) +  p(\Gamma,1)^{-1}(C) \cap \partial C(M_{[0,4]},\Gamma) \cap \bigcap_{i=2}^{3n}H_i(\Gamma).$$

Let us check that the first sign is correct. The two other ones are similar.
The normal bundle of the one-manifold reads $N^1_C \oplus \bigoplus_{i=2}^{3n} N^i$ where $N^1_C$ is the normal bundle to $p(\Gamma,1)^{-1}(C)$ in $C(M,\Gamma)$.
Since $H^{\prime}_1$ is oriented as the boundary of $C$, $C$ is oriented as the outward normal $N_o$ to $C$, followed by the orientation of $H^{\prime}_1$ along $H^{\prime}_1$. Therefore the normal bundle of
$p(\Gamma,1)^{-1}(H^{\prime}_1)$ is oriented by $N^1_C \oplus N_o$.

Use the notation
$\langle \frac{1}{\hat{\delta}} H_1(\Gamma), \dots,\frac{1}{\hat{\delta}} H_{3n}(\Gamma)\rangle_{\CA,C(M,\Gamma)}$,
where $\hat{\delta}=k\delta(t)$,
to denote $I_{\Gamma}(\{H_i\})$, that was defined before the statement of Theorem~\ref{thmhighloop}.
With similar notation, set
$$I_{\Gamma}(C,\{H_i\})=- \langle \frac{1}{\hat{\delta}}p(\Gamma,1)^{-1}(C) \cap \partial C(M_{[0,4]},\Gamma) , \frac{1}{\hat{\delta}}H_2(\Gamma) ,\dots, \frac{1}{\hat{\delta}}H_{3n}(\Gamma) \rangle_{\CA,C(M,\Gamma)}.$$
Then $I_{\Gamma}(C,\{H_i\})$ is the variation of $I_{\Gamma}(\{H_i\})$ when $H_1$ is changed to $H^{\prime}_1$.
The intersection points in $p(\Gamma,1)^{-1}(C) \cap \partial C(M_{[0,4]},\Gamma) \cap H_2(\Gamma) \cap \dots \cap H_{3n}(\Gamma)$ only occur inside codimension $1$ faces of $C(M_{[0,4]},\Gamma)$.

\begin{lemma}
 All these intersection points project to $M_{[0,4[}^{2n}$ under $p(M,\Gamma)$.
\end{lemma}
\begin{proof}
Consider the projection $(m_i)_{i=1,2,\dots,2n}$ of such an intersection point.
Assume that $r(m_{\ell})>4-\varepsilon$.
Construct a subgraph $G$ of $\Gamma$ such that the edges of $G$ have an orientation independent of the edge orientation in $\Gamma$, so that
$G$ is maximal with the properties:
$G$ contains $v(\ell)$, there is an increasing path (following edge orientations in $G$) from $v(\ell)$ to every other vertex of $G$, if an edge $e(a)$ of $G$ goes from $v(i)$ to $v(j)$ then
$r(m_j)>r(m_i)-2\varepsilon$.

Such a graph $G$ can be constructed inductively starting from $v(\ell)$:
For each constructed vertex $v(i)$ with less than three incoming edges, add each other edge of $\Gamma$ adjacent to $v(i)$ whose other end $v(j)$ satisfies $r(m_j)>r(m_i)-2\varepsilon$ and orient it from $v(i)$ to $v(j)$.

For every $v(i) \in G$, $r(m_{i})>2$.

In $G$, the valency of a vertex that does not belong to $e(1)$ is $2$ or $3$ while the valency of a vertex that belongs to $e(1)$ is at least $1$.
Indeed, if the edge $e(a\neq 1)=\{v(i),v(j)\}$ of $\Gamma$ is not in $G$, and if $v(i) \in G$, then $r(m_j)\leq r(m_i)-2\varepsilon$, and $\overline{\pi}(m_i,m_j)=\pm X_a$ so that $m_i \in P_V(X_a)$.
Therefore, since $(X_a)_a$ is regular, we cannot erase more than one edge different from $e(1)$ that contains $v(i)$ from $\Gamma$.

Now, look at the constraints we know for the elements $m_j \in M_{[2,4]}$ for which $v(j) \in G$. They read $\overline{\pi}(m_i,m_j)=\pm X_b$ for internal edges $e(b \neq 1)=\{v(i),v(j)\}$ of $G$, and $m_j \in P_V(X_a)$ for other edges $e(a \neq 1)$ of $\Gamma$ that contain $v(j)$.
The vertical translation induces an action of $S^1$ on the space $C_{V(G)}(M_{[2,4]})$ of configurations of the vertices of $G$ in $M_{[2,4]}$, and the constraint maps are invariant under this action. Then if no vertex of $G$ belongs to $e(1)$, the product of the above elementary constraints maps, that only depend on the restriction of the configuration to $V(G)$, goes to a space whose dimension is the number of half-edges of $\Gamma$ that contain an element of $G$, that is the dimension of $M_{[2,4]}^{V(G)}$, that is greater than the dimension of $C_{V(G)}(M_{[2,4]})/S^1$. 

When a vertex of $G$ belongs to $e(1)$, and when the above map is a submersion, it becomes a local diffeomorphism on $C_{V(G)}(M_{[2,4]})/S^1$.
Then there are isolated preimages in $C_{V(G)}(M_{[2,4]})/S^1$ that provide isolated $S^1$-orbits of $C_{V(G)}(M_{[2,4]})$ that must lie in the interior of $C_{V(G)}(M_{[2,4]})$. Then $r(m_{\ell})<4$.
\end{proof}
Thus intersection points in $p(\Gamma,1)^{-1}(C) \cap \partial C(M_{[0,4]};\Gamma) \cap H_2(\Gamma) \cap \dots \cap H_{3n}(\Gamma)$ only occur inside codimension $1$ faces of $C(M,\Gamma)$.

Let us study some point $P$ in this intersection.
Its codimension one face corresponds to some connected
subgraph $\Gamma_I$ made of the vertices numbered in a set $I$ and the edges between two of these, as in Subsection~\ref{subcompconf}.
Define the graph $\Gamma(I)=\Gamma/\Gamma_I$ obtained from $\Gamma$ by identifying $\Gamma_I$ with a point $v(I)$. The vertices of $\Gamma(I)$ are the vertices of $\Gamma$ numbered in $\{1, \dots, 2n\} \setminus I$ and the additional vertex $v(I)$, and the edges of $\Gamma(I)$ are the edges of $\Gamma$ that are not edges of $\Gamma_I$. The set $E(\Gamma)$ of edges of $\Gamma$ is the disjoint union of the sets $E(\Gamma_I)$ and $E(\Gamma(I))$.

Our face $\CF(\Gamma,I)$ fibers over the configuration space of the vertices of $\Gamma(I)$ that form the set $V(\Gamma(I))$, and our intersection point projects to $M_{[0,4[}^{V(\Gamma(I))}$. The fiber is made of the infinitesimal configurations of $I$ up to translation and dilation.
The dimension of the fiber is $(3\sharp I -4)$ and 
the dimension of the base is $3(\sharp V(\Gamma(I)))= 6n - 3\sharp I +3$.

Since the intersection is transverse, the edges give $(6n-1)$ independent constraints on our intersection point $P$, where the edge numbered by $1$ gives one of these constraints and each of the other ones gives two constraints.

The constraints coming from the edges of $\Gamma(I)$ only concern the base, while the constraints of $E(\Gamma_I)$ only concern the fiber
except possibly for one edge $(e(j) \in E(\Gamma_I))$ where the constraint can read: 
``The point $v(I) \in K_j$.''

In this case, this constraint is a constraint for the base, and $j\neq 1$.
If we have such an exceptional constraint ``$v(I) \in K_j$'', set $\chi_k=1$. Otherwise, set $\chi_k=0$.

If the edge $e(1)$ belongs to $E(\Gamma_I)$, set $\chi_A=1$, and if it belongs to 
$E(\Gamma(I))$, set $\chi_A=0$.

Then we have $2 \sharp E(\Gamma_I) - 2\chi_k - \chi_A$ constraints on the $(3\sharp I -4)$--dimensional fiber.

The transversality condition tells us that 
$2 \sharp E(\Gamma_I) - 2\chi_k - \chi_A = 3\sharp I - 4$
that is $$4 -2\chi_k-\chi_A=3\sharp I-2\sharp E(\Gamma_I)$$
where the right-hand side is the valency of the vertex $v(I)$ in $\Gamma(I)$ as a count of half-edges shows.
This valency must be 4, 3, 2 or 1.

The transversality also tells us that the map ``direction of the edges'' from the $(3\sharp I -4)$--dimensional fiber to the product of the images
of the constraint functions must be a local diffeomorphism.

In particular, unless $\sharp I=2$, there can be no vertex of valency $1$ in $\Gamma_I$ because moving this vertex in the direction
of the unique edge incident to such a vertex would not change the image of the above map.

Thus, either $\Gamma_I$ is an edge, and this case will be treated by the IHX (or Jacobi) identification,
or $\Gamma_I$ is connected, all the vertices of $\Gamma_I$ have valency
$2$ or $3$ and there are $\chi$ vertices of valency $2$ where
$1 \leq \chi \leq 4$, and this case will be treated by the parallelogram identification.

Let us explain how these classical identifications allow us to prove that the sum of the contributions of all such intersection points to 
$$\CV(C,\{H_i\})=\sum_{\Gamma \in \CS_n} I_{\Gamma}(C,\{H_i\})$$ vanishes in our setting. We start with the parallelogram case.

\begin{lemma}
Let $E$ be the set of codimension one faces $\CF(\Gamma,I)$ of spaces $C(M_{[0,4[},\Gamma)$ indexed by
some $(\Gamma,I)$ where $\Gamma \in \CS_n$, $I$ is a subset of $\{1,2,\dots,2n\}$, $\Gamma_I$ is connected, the vertices of $\Gamma_I$ have valency
$2$ or $3$ and at least one vertex of $\Gamma_I$ has valency $2$. 
There is an involution $\hat{\rho}$ of $E$
without fixed point
such that the sum of the contributions of the intersection points in two faces $\CF_1$ and $\hat{\rho}(\CF_1)$ to
$\CV(C,\{H_i\})$ vanishes.
\end{lemma}
\begin{proof} Fix $\Gamma$ and $I$ as above.
Let $i$ be the smallest element of $I$ such that the valency of $v(i)$ in $\Gamma_I$ is $2$. Define the labelled graph $\rho(\Gamma,I) \in \CS_n$ from $\Gamma$ by modifying the
two edges $e(a)$ and $e(b)$ incident to $v(i)$ in $\Gamma_I$ as follows:  if $e(a)$ reads $\varepsilon(a)\overrightarrow{v(i)v(j)}$ and if $e(b)$ reads $\varepsilon(b)\overrightarrow{v(i)v(k)}$ in $\Gamma$, then $e(a)$ reads $\varepsilon(a)\overrightarrow{v(k)v(i)}$ and $e(b)$ reads $\varepsilon(b)\overrightarrow{v(j)v(i)}$ in $\rho(\Gamma,I)$. Define $\hat{\rho}$ so that $$\hat{\rho}(\CF(\Gamma,I))=\CF(\rho(\Gamma,I),I).$$
Consider the orientation-reversing involutive diffeomorphism $\phi(\hat{\rho})$ from $\CF(\Gamma,I)$ to $\hat{\rho}(\CF(\Gamma,I))$ that maps a point $P$ to the point where the only changed piece of data is the position of $v(i)$ in the infinitesimal configuration of $I$ where $v(i)$ is mapped to its symmetric with respect to the middle of the two (possibly coinciding)
other ends of $e(a)$ and $e(b)$.
The standard properties of a parallelogram guarantee that for any edge label $j$, $p(\Gamma,j)=p(\rho(\Gamma,I),j) \circ \phi(\hat{\rho})$ on $\CF(\Gamma,I)$.

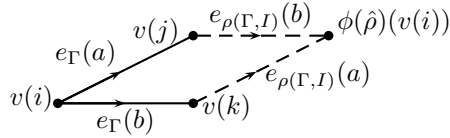
$$\begin{pspicture}[shift=-0.2](0,0)(4,1.5)
\psset{xunit=.6cm,yunit=.6cm}
\rput[b](5.5,2.1){$e_{\rho(\Gamma,I)}(b)$}
\rput[t](2.5,.4){$e_{\Gamma}(b)$}
\rput[l](5.6,1.1){$e_{\rho(\Gamma,I)}(a)$}
\rput[r](2.3,1.6){$e_{\Gamma}(a)$}
\rput[r](.9,.6){$v(i)$}
\rput[r](3.7,2.1){$v(j)$}
\rput[l](4.2,.4){$v(k)$}
\rput[l](7.2,2.3){$\phi(\hat{\rho})(v(i))$}
\psline[linestyle=dashed]{->}(4,.5)(5.4,1.2)
\psline[linestyle=dashed]{-}(5.4,1.2)(7,2)
\psline{-}(1,.5)(4,2)
\psline{->}(1,.5)(2.4,1.2)
\psline[linestyle=dashed]{*->}(4,2)(5.5,2)
\psline[linestyle=dashed]{-*}(5.5,2)(7,2)
\psline{-*}(1,.5)(4,.5)
\psline{*->}(1,.5)(2.5,.5)
\end{pspicture}$$
Consider a point $P$ of $\CF(\Gamma,I)$ that contributes to $I_{\Gamma}(C,\{H_i\})$.
Recall that the chains $C$ and $H_a \cup k\delta(t) \pi^{-1}(X_a)$ are transverse to $\partial C_2(M_{[0,4[})$. They are made of two kinds of pieces there: $(-J_{\Delta}) ST(M)_{|K_a}$ and 
$s_{\tau}(M;P(X_1,X^{\prime}_1))$
or $s_{\tau}(M;X_a)$. Since our assumptions imply that the intersection is transverse, if $\sharp I > 2$, $P$ cannot be in the parts $p(\Gamma,a)^{-1}(ST(M)_{|K_a})$ or $p(\Gamma,b)^{-1}(ST(M)_{|K_b})$ because in this case the constraint coming from the other edge adjacent to $P$ gives the direction of this other edge and stretching this other edge leads to a non trivial kernel of the constraint map.
Therefore, if $\sharp I > 2$, the two edges can be assumed to be weighted by $1=\frac{k\delta(t)}{k\delta(t)}$.
Thus, when a point $P$ as above contributes to the variation $I_{\Gamma}(C,\{H_i\})$, the point $\phi(\hat{\rho})(P)$ contributes with the same beaded graph with the opposite sign to $I_{\rho(\Gamma,I)}(C,\{H_i\})$.
(Indeed, the product of the constraint maps only differs by the orientation-reversing diffeomorphism $\phi(\hat{\rho})$ of the faces, and the graphs are obtained from one another by an exchange of two edge labels and a possible simultaneous reverse of two edge orientations, for two edges beaded by $1$.)
If $\sharp I=2$, our assumptions imply that $\Gamma_I$ is made of two edges $e(a)$ and $e(b)$ between the two elements $i$ and $j$ of $I$ like in the figure.
$$
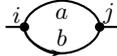
\begin{pspicture}[shift=-0.5](-.1,-.4)(2,.4)
\psset{xunit=.5cm,yunit=.5cm}
\psline{-*}(0,0)(.5,0)
\psline{*-}(2.5,0)(3,0)
\psecurve{-}(1.5,-.7)(.5,0)(1.5,.7)(2.5,0)(1.5,-.7)
\psecurve{-}(1.5,.7)(.5,0)(1.5,-.7)(2.5,0)(1.5,.7)
\psecurve{->}(1.5,.7)(.5,0)(1.5,-.7)(2.5,0)
\rput[b](1.5,-.5){\small $b$}
\rput[t](1.5,.5){\small $a$}
\rput[br](.4,.1){\small $i$}
\rput[bl](2.6,.1){\small $j$}
\end{pspicture}$$
In $\rho(\Gamma,I)$, the orientations of both edges are reversed. Up to exchanging $a$ and $b$, $P \in p(\Gamma,a)^{-1}(s_{\tau}(M;X_a))  \cap (-J_{\Delta}) p(\Gamma,b)^{-1}(ST(M)_{|K_b})$, and the rational functions are $\pm 1$ and $\pm J_{\Delta}$.

Let us show that we have the same cancellation as before in this case.
Let $K_b \times D_b$ denote a trivialised tubular neighborhood of $K_b$.
Then the local constraint map associated to $(-1/(J_{\Delta}\hat{\delta}))H_b$ near $ST(M)_{|K_b}$
can be viewed as a map $f^{\prime}_b\colon ST(M) \rightarrow D_b$ associated with the natural projection of $K_b \times D_b$ to the disk $D_b$.

Let $P^{\prime}$ be a transverse intersection point in $$p(\Gamma,a)^{-1}(s_{\tau}(M;X_a))  \cap p(\Gamma,b)^{-1}(ST(M)_{|K_b}).$$
Let $\eta$ be the local degree  of $(\prod_{e\neq b} f_e \circ p(\Gamma,e),f^{\prime}_b \circ p(\Gamma,b))_{|\CF(\Gamma,I)}$ at $P^{\prime}$, where the $f_e$ are local constraint maps associated to the $H_e\cup k\delta(t) \pi^{-1}(X_e)$. Then
$P^{\prime}$ contributes as $$\eta
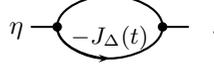
\begin{pspicture}[shift=-0.5](-.1,-.6)(2.4,.6)
\psset{xunit=.7cm,yunit=.6cm}
\psline{-*}(0,0)(.5,0)
\psline{*-}(2.5,0)(3,0)
\psccurve(.5,0)(1.5,-.7)(2.5,0)(1.5,.7)
\psecurve{->}(1.5,.7)(.5,0)(1.5,-.7)(2.5,0)
\rput[b](1.5,-.5){\small $-J_{\Delta}(t)$}
\end{pspicture}.$$
As before, the local degree of 
$(\prod_{e\neq b} f_e \circ p(\rho(\Gamma,I),e),
f^{\prime}_b \circ p(\rho(\Gamma,I),b))_{|\CF(\rho(\Gamma,I),I)}$ at $\phi(\hat{\rho})(P^{\prime})$ is $(-\eta)$  since this map is the composition of the former one by $\phi(\hat{\rho})^{-1}$.

Let $\iota$ denote the orientation-reversing involution of $C_2(M)$ that extends the exchange of points in $(M^2\setminus \mbox{diag})$ and let $\iota_{S^2}$ denote the antipode of $S^2$.

Focus on the edges $e(a)$ and $e(b)$. Without loss, consider $\eta$
as the degree of 
$$f_a \times f^{\prime}_b \colon (ST(D_b)=_{\tau}S^2 \times D_b) \rightarrow S^2 \times D_b$$
where $(ST(M)=\partial C_2(M))$ and $ST(D_b)$ are oriented by the orientation of $e_{\Gamma}(b)$, and $f_a$ is the projection of $ST(D_b)$ on $S^2$ if the orientations of $e_{\Gamma}(a)$ and $e_{\Gamma}(b)$ coincide, and its composition by $\iota_{S^2}$ otherwise.

With these orientations and conventions, $e_{\rho(\Gamma,I)}(b)$ is equipped with 
$$-J_{\Delta}(t)\iota(ST(M)_{|K_b})=-J_{\Delta}(t^{-1})ST(M)_{|K_b}$$
while $(-\eta)$ is the degree of $(f_a\circ \iota) \times (f^{\prime}_b\circ \iota)=(\iota_{S^2}\circ f_a)\times f^{\prime}_b.$
Then $\phi(\hat{\rho})(P^{\prime})$ contributes as
$(-\eta) \begin{pspicture}[shift=-0.95](-.1,-1)(2,.6)
\psset{xunit=.6cm,yunit=.6cm}
\psline{-*}(0,0)(.5,0)
\psline{*-}(2.5,0)(3,0)
\psccurve(.5,0)(1.5,-.7)(2.5,0)(1.5,.7)
\psecurve{<-}(.5,0)(1.5,-.7)(2.5,0)(1.5,.7)
\rput[t](1.5,-.8){\small $-J_{\Delta}(t^{-1})$}
\end{pspicture}
=(-\eta)\begin{pspicture}[shift=-0.95](-.1,-1)(2.4,.6)
\psset{xunit=.7cm,yunit=.6cm}
\psline{-*}(0,0)(.5,0)
\psline{*-}(2.5,0)(3,0)
\psccurve(.5,0)(1.5,-.7)(2.5,0)(1.5,.7)
\psecurve{->}(1.5,.7)(.5,0)(1.5,-.7)(2.5,0)
\rput[b](1.5,-.5){\small $-J_{\Delta}(t)$}
\end{pspicture}$.
\end{proof}

\begin{lemma}
 The intersection points of the faces $\CF(\Gamma,I)$ where $\Gamma_I$ is an edge do not make the sum $\tilde{z}_n(M,\KK,\tau)$ vary, thanks to the IHX relation.
\end{lemma}
\begin{proof} Let $f$ be the label of the unique edge of $\Gamma_I$.
Like in the previous proof, notice that an intersection point $P$ in such a face cannot come
from the $p(\Gamma,f)^{-1}(ST(M)_{|K_f})$ part. In particular, the edge $e(f)=\Gamma_I$ is beaded by $1$.
Consider the graph $\Gamma(I)=\Gamma/\Gamma_I$. Its vertex $v(I)$ has $4$ incident edges $e(a)$, $e(b)$, $e(c)$, $e(d)$, and, with the natural identifications, in the initial graph $\Gamma$, $e(f)$ goes from $v(i)$ to $v(j)$
where $v(i)$ is incident to two edges among $e(a)$, $e(b)$, $e(c)$, $e(d)$ and $v(j)$ is incident to the two other ones. 
For $g=b$, $c$ or $d$, let $\Gamma^g$ be the graph of $\CS_n$ such that $\Gamma^g(I=\{i,j\})=\Gamma(I)$, and $v(i)$ is incident to $e(a)$ and $e(g)$.
It is enough to prove that the contributions of the possible intersection points of the three faces indexed by $(\Gamma^g,I)$ cancel in the sum $\CV(C,\{H_i\})$.
Our intersection point $P$ of $C(M,\Gamma)$ will give a transverse intersection point in the three faces, with a common sign, and $\Gamma^b$, $\Gamma^c$ and $\Gamma^d$ are oriented so that
their vertex orientations coincide outside the disk of the following picture, and they are given by the picture for the two vertices there.
Indeed the roles of the half-edges indexed by $b$, $c$ and $d$ are cyclically permuted.
$$
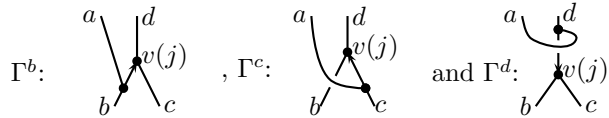
\begin{pspicture}[shift=-0.2](-1,-.2)(1.6,1.2)
\psset{xunit=1.2cm,yunit=1.2cm}
\rput[l](-.9,.4){$\Gamma^b$:}
\rput[r](.05,1){$a$}
\rput[r](.2,0){$b$}
\rput[l](.8,0){$c$}
\rput[l](.55,1){$d$}
\rput[l](.55,.55){$v(j)$}
\psline{-*}(.1,1)(.35,.2)
\psline{*-}(.5,.5)(.5,1)
\psline{-}(.75,0)(.5,.5)
\psline{->}(.25,0)(.5,.5)
\end{pspicture}
\begin{pspicture}[shift=-0.2](-1.2,-.2)(1.6,1.2)
\psset{xunit=1.2cm,yunit=1.2cm}
\rput[l](-.9,.4){, $\Gamma^c$:}
\rput[r](.05,1){$a$}
\rput[r](.2,0){$b$}
\rput[l](.8,0){$c$}
\rput[l](.55,1){$d$}
\rput[l](.55,.65){$v(j)$}
\psline{*-}(.5,.6)(.5,1)
\psline{->}(.8,0)(.5,.6)
\psline{-}(.2,0)(.5,.6)
\pscurve[border=2pt]{-*}(.1,1)(.3,.3)(.7,.2)
\end{pspicture}
\begin{pspicture}[shift=-0.2](-1.2,-.2)(1.6,1.2)
\psset{xunit=1.2cm,yunit=1.2cm}
\rput[l](-.9,.4){and $\Gamma^d$:}
\rput[r](.05,1){$a$}
\rput[r](.2,0){$b$}
\rput[l](.8,0){$c$}
\rput[l](.55,1.05){$d$}
\rput[l](.55,.4){$v(j)$}
\psline{<-}(.5,.35)(.5,1)
\psline{-*}(.75,0)(.5,.35)
\psline{-}(.25,0)(.5,.35)
\pscurve[border=2pt]{-*}(.1,1)(.2,.75)(.7,.75)(.5,.85)
\end{pspicture}
$$
\end{proof}

Let us now prove that $\tilde{z}_n((M,M_{[1,8]}),\KK,\tau,\funcM,\varepsilon,(K_i)_i)$ does not depend on $\funcM$.

\begin{proof}
Let $\dfuncM \colon M \rightarrow S^1$ be a map that coincides with $\funcM$ on $M_{[1,8]}$. Then the graphs $G(\funcM)$ and $G(\dfuncM)$ of the restrictions of $\funcM$ and $\dfuncM$ to $M_{[0,1]}$ cobound a $4$-chain $\CfuncM$ in $M_{[0,1]}\times S^1$. $\partial \CfuncM =G(\dfuncM) -G(\funcM)$.
The intersection of $\pi(\funcM)^{-1}(X_i)$ with $M_{[0,1]} \times M_{[2,4]}$
reads $\{(m,(\projq(t,z_W(X_i)),z(t)z(X_i)\funcM(m))); m \in M_{[0,1]}, t \in [2,4]\}$.
Let $C(\CfuncM,X_i)=A(\CfuncM,X_i) - B(\CfuncM,X_i)$ where \\
$A(\CfuncM,X_i)=\{(m,(\projq(t,z_W(X_i)),z(t)z(X_i)c)); (m,c) \in \CfuncM, t \in [1.5,4]\}$, and\\
$B(\CfuncM,X_i) = \{((\projq(t,-z_W(X_i)),(z(t)z(X_i))^{-1}c),m); (m,c) \in \CfuncM, t \in [1.5,4]\}$ oriented by $(m,c,t)$.
Define $H^{\prime}_i$ so that $H^{\prime}_i=(H_i + \partial C(\CfuncM,X_i))\cap C_2(M_{[0,2]})$.
As before, we can assume that, for any $i \in \{1,\dots, 3n\}$, the $H^{\prime}_j(\Gamma)$ for $j<i$, $p(\Gamma,i)^{-1}(C(\CfuncM,X_i))$ and the $H_j(\Gamma)$ for $j>i$, are transverse. Then the sum over $\Gamma \in \CS_n$ of the variations
$$\langle \frac{1}{\hat{\delta}}H^{\prime}_1,\dots, \frac{1}{\hat{\delta}}H^{\prime}_{i-1},-\frac{1}{\hat{\delta}}p(\Gamma,i)^{-1}C(\CfuncM,X_i) \cap \partial C(M_{[0,4]},\Gamma) , \frac{1}{\hat{\delta}}H_{i+1} ,\dots, \frac{1}{\hat{\delta}}H_{3n} \rangle_{\CA}$$
(where $H^{\prime}_j$ and $H_j$ stand for $H^{\prime}_j(\Gamma)$ and $H_j(\Gamma)$)
vanishes as above for any $i \in \{1,\dots, 3n\}$, and $\tilde{z}_n$ does not depend on the choice of $\funcM$.
\end{proof}

Now, $\tilde{z}_n(M,\KK,\tau,\varepsilon,(K_i)_i)$ is a locally constant function of $(\varepsilon,(K_i)_i)$ that is therefore a topological invariant of $(M,\KK,\tau)$.

\subsection{The dependence on the trivialisation}

We now compute the variation of $\tilde{z}_n(M,\KK,\tau)$ under a change of trivialisation.
Let $\tau^{\prime}$ be a trivialisation that coincides with $\tau$ on the tubular neighborhood $M_{[1,8]}$ of $K$.
We want to compute the variation caused by the replacement of the  $H_i$ by chains $H^{\prime}_i$ of $C_2(M_{[0,2]})$ where
$\partial H^{\prime}_i = k\delta(t)(s_{\tau^{\prime}}(M;X_i) \cup \pi_{|\partial \tilde{C}_2(M_{[0,2]}) \setminus \partial \tilde{C}_2(M_{[0,2[})}^{-1}(X_i)-J_{\Delta}(t)ST(M)_{|K_i})$.
There exists a cobordism $c_i$ between $s_{\tau}(M_{[0,1]};X_i)$ and $s_{\tau^{\prime}}(M_{[0,1]};X_i)$ in $ST(M_{[0,1]})$ \cite[Lemma 5.10]{betaone}, and there is a $5$-chain $C_i$ of $C_2(M_{[0,2]})$ such that 
$$\partial C_i = H^{\prime}_i-H_i \cup  k\delta(t)c_i.$$
Then $$\tilde{z}_n(M,\KK,\tau^{\prime})-\tilde{z}_n(M,\KK,\tau)=\sum_{\Gamma \in \CS_n}\sum_{i=1}^{3n}\CV(\Gamma,i)$$ where $\CV(\Gamma,i)$ equals
$$\langle \frac{1}{\hat{\delta}}H^{\prime}_1,\dots, \frac{1}{\hat{\delta}}H^{\prime}_{i-1},-\frac{1}{\hat{\delta}}p(\Gamma,i)^{-1}(C_i) \cap \partial C(M_{[0,4]},\Gamma) , \frac{1}{\hat{\delta}}H_{i+1} ,\dots, \frac{1}{\hat{\delta}}H_{3n} \rangle_{\CA,C(M,\Gamma)}$$
like in the previous subsection. Now the proof of the previous subsection can be applied verbatim to cancel most of the contributions, except for the computation of the valency of $v(I)$ in $\Gamma(I)$, because if the tangent space of a point of $M$ is identified with $\RR^3$ via $\tau$,  constraints associated with edges of $\Gamma_I$ coming from chains $H^{\prime}_j$ or $C_i$ constrain both the fiber and the base. Therefore, the valency of $v(I)$ can be zero and this is the only case that is not covered by the cancellations of the previous subsection. In this case, $\Gamma_I=\Gamma$. Let $\CF(\Gamma)$ denote $\CF(\Gamma,\{1,\dots,2n\})$. Then $\CV(\Gamma,i)$ reads
$$\pm\langle s_{\tau^{\prime}}(M;X_1),\dots, s_{\tau^{\prime}}(M;X_{i-1}),p(\Gamma,i)^{-1}(c_i), s_{\tau}(M;X_{i+1}) ,\dots, s_{\tau}(M;X_{3n}) \rangle_{\CA,\CF(\Gamma)}.$$
In particular, $(\tilde{z}_n(M,\KK,\tau^{\prime})-\tilde{z}_n(M,\KK,\tau))$, that does not involve any coefficient or non-trivial bead, is the same as the variation in the case of the
universal finite type invariant of rational homology spheres, that is $(p_1(\tau)-p_1(\tau^{\prime}))\frac{\xi_n}{4}$ like in \cite[Section 1.6]{lesconst}.
This finishes the proof of Theorem~\ref{thmhighloop}.

\section*{Appendix: A configuration space definition of the Casson-Walker invariant}
\label{subdefcas}

For $r\in \RR$, let $B(r)$ denote the ball of radius $r$ in $\RR^3$
that is equipped with its standard parallelisation $\tau_s$.
A rational homology sphere $N$ may be written as $B_N \cup_{B(1) \setminus \mathring{B}(1/2)} B^3$ where
$B_N$ is a {\em rational homology ball\/}, that is a connected compact (oriented) smooth $3$--manifold with boundary $S^2$ with the same rational homology as a point, $B^3$ is a $3$-ball, $B_N$ contains $(B(1) \setminus \mathring{B}(1/2))$ as a neighborhood of its boundary $\partial B_N=\partial B(1)$, and $B^3$ contains $(B(1) \setminus\mathring{B}(1/2))$ as a neighborhood of its boundary $\partial B^3=-\partial B(1/2)$.
Let $B(N)=B_N(3)$ be obtained from $B(3)$ by replacing the unit ball $B(1)$ of $\RR^3$ by $B_N$. Equip $B(N)$ with a trivialisation $\tau_N$ that coincides with $\tau_s$ outside $B_N$.

Let $W$ be a compact connected $4$-manifold with signature $0$ and with boundary 
$$\partial W =B_N(3) \cup_{\{1\} \times \partial B(3)} \left(- [0,1] \times \partial B(3) \right) \cup_{\{0\} \times \partial B(3)} (-B(3)).$$
Define $p_1(\tau_N) \in (H^4(W,\partial W;\pi_3(SU(4))) = \ZZ)$ as the obstruction to extend the trivialisation of $TW \otimes \CC$
induced by $\tau_s$ and $\tau_N$ on $\partial W$ to $W$. 
Again, we use the notation and conventions of \cite{milnorsta}, see also \cite[Section 1.5]{lesconst}.

Let $\RR^3(N)$ be obtained from $\RR^3$ by replacing its unit ball $B(1)$ by $B_N$. Let $C_2(\RR^3(N))$ be obtained from $\RR^3(N)^2$ by blowing up the diagonal as in Definition~\ref{defblowup}. Let $P \colon C_2(\RR^3(N)) \rightarrow \RR^3(N)^2$ be the associated canonical projection and let $C_2(B(N))=P^{-1}(B(N)^2)$.
Consider a smooth map 
$ \chi \colon \RR \rightarrow [0,1]$ that maps $]-\infty,-2]$ to $0$ and
$[-1,\infty[$ to $1$.
Define $$\begin{array}{llll}p_{B(3)} \colon &B(3)^2\setminus \mbox{diagonal}& \rightarrow &S^2\\
&(U,V) &\mapsto &\frac{\chi(\parallel V\parallel-\parallel U \parallel)V-\chi(\parallel U\parallel-\parallel V \parallel)U}{\parallel \chi(\parallel V\parallel-\parallel U \parallel)V-\chi(\parallel U\parallel-\parallel V \parallel)U \parallel} \end{array}
$$
This map extends to $C_2(B(3))$ to a map still denoted by $p_{B(3)}$, that reads as the projection to $S^2$ induced by $\tau_s$ (see Subsection~\ref{subconst}) on the unit tangent bundle of $B(3)$.
A similar map $p_N$ can be defined on the boundary $\partial C_2(B(N))$: The map $p_N$ is the projection to $S^2$ induced by $\tau_N$ on the unit tangent bundle of $B(N)$, and the map $p_N$ is 
given by the above formula, where we set $\parallel U\parallel=1$ when $U \in B_N$, for the other points of the boundary that are pairs
$(U,V)$ of $\left(B(N)^2\setminus \mbox{diagonal}\right)$ where $U$ or $V$ belongs to $\partial B(3)$ (therefore a possible point of $B_N$ is replaced by $0 \in \RR^3$ in the formula).

The following theorem, that gives a configuration space definition for the Casson-Walker invariant, is due to Kuperberg and Thurston \cite{kt} for the case of integral homology spheres (though it is stated in other words). It has been generalised to rational homology spheres in \cite[Section 6]{sumgen}.

\begin{theorem}
\label{thmdefcasconf}
Let $\fvarM$, $\svarM$ and $\tvarM$ be three distinct points of $S^2$.
Under the above assumptions, for $\cvarM=\fvarM$, $\svarM$ or $\tvarM$, the submanifold $p_N^{-1}(\cvarM)$ of $\partial C_2(B(N))$ bounds a rational chain $F_{N,\cvarM}$
in $C_2(B(N))$, and
$$\lambda(N)=\frac{\langle F_{N,\fvarM},F_{N,\svarM}, F_{N,\tvarM}\rangle_{C_2(B(N))}}{6} -\frac{p_1(\tau_N)}{24}.$$
\end{theorem}

It is easy to see that $\langle F_{N,\fvarM},F_{N,\svarM}, F_{N,\tvarM}\rangle_{C_2(B(N))}$ is a well-defined invariant of $(N,\tau_N)$.
Indeed,  $C_2(B(N))$ has the same rational homology as $\left((\RR^3)^2 \setminus \mbox{diagonal}\right)$ that is homotopy equivalent to $S^2$ via the map $(x,y) \mapsto \frac{y-x}{\parallel y-x\parallel}$.
Therefore, $C_2(B(N))$ has the same rational homology as $S^2$.
In particular, since $H_3(C_2(B(N));\QQ)=\{0\}$, the cycle $p_N^{-1}(\cvarM)$ bounds a rational chain in $C_2(B(N))$ and, since
$$H_4(C_2(B(N));\QQ)=\{0\},$$ $\langle F_{N,\fvarM},F_{N,\svarM}, F_{N,\tvarM}\rangle_{C_2(B(N))}$ only depends on the non-intersecting
boundaries of the $F_{N,\cvarM}$.

\bibliographystyle{amsplain}

\end{document}